\documentclass[reqno]{amsart}
\usepackage{amssymb}
\usepackage{graphicx}

\usepackage[usenames, dvipsnames]{color}
\usepackage{verbatim}
\usepackage{mathrsfs}
\usepackage{bm}
\usepackage{cite}

\numberwithin{equation}{section}

\newtheorem{theorem}{Theorem}[section]

\newtheorem{lemma}[theorem]{Lemma}
\newtheorem{prop}[theorem]{Proposition}

\theoremstyle{definition}
\newtheorem{remark}[theorem]{Remark}

\theoremstyle{definition}

\theoremstyle{definition}

\makeatletter
\def\dashint{\operatorname%
{\,\,\text{\bf-}\kern-.98em\DOTSI\intop\ilimits@\!\!}}
\makeatother

\def\\det{\text{\det}}

\def\Xint#1{\mathchoice
 {\XXint\displaystyle\textstyle{#1}}%
 {\XXint\textstyle\scriptstyle{#1}}%
 {\XXint\scriptstyle\scriptscriptstyle{#1}}%
 {\XXint\scriptscriptstyle\scriptscriptstyle{#1}}%
 \!\int}
\def\XXint#1#2#3{{\setbox0=\hbox{$#1{#2#3}{\int}$}
  \vcenter{\hbox{$#2#3$}}\kern-.5\wd0}}

\def\dashint{\Xint-}

\def\.5{\frac{1}{2}}

\newcommand{\RN}[1]{%
  \textup{\uppercase\expandafter{\romannumeral#1}}%
}

\renewcommand{\epsilon}{\varepsilon}

\newcounter{marnote}

\begin{document}

\title[Local behavior for weighted parabolic equations]{Local behavior for solutions to anisotropic weighted quasilinear degenerate parabolic equations}

\author[C.X. Miao]{Changxing Miao}
\address[C.X. Miao] {1. Beijing Computational Science Research Center, Beijing 100193, China.}
\address{2. Institute of Applied Physics and Computational Mathematics, P.O. Box 8009, Beijing, 100088, China.}
\email{miao\_changxing@iapcm.ac.cn}

\author[Z.W. Zhao]{Zhiwen Zhao}

\address[Z.W. Zhao]{Beijing Computational Science Research Center, Beijing 100193, China.}
\email{zwzhao365@163.com}

%\footnote{}

\date{\today} % delete this line to display the current date

%%% BEGIN DOCUMENT

\maketitle
%\tableofcontents
\begin{abstract}
This paper aims to study the local behavior of solutions to a class of anisotropic weighted quasilinear degenerate parabolic equations with the weights comprising two power-type weights of different dimensions. We first capture the asymptotic behavior of the solution near the singular or degenerate point of the weights. In particular, we find an explicit upper bound on the decay rate exponent determined by the structures of the equations and weights, which can be achieved under certain condition and meanwhile reflects the damage effect of the weights on the regularity of the solution. Furthermore, we prove the local H\"{o}lder regularity of solutions to non-homogeneous parabolic $p$-Laplace equations with single power-type weights.
\end{abstract}

\maketitle
%\date{}
%\maketitle
%{\bf Abstract}

%{Keywords}

%\noindent{\bf{MSC numbers}}: {35K92; 35B40; 35B65.}

\section{Introduction and main results}%\label{intro}

Let $\Omega$ be a smooth bounded domain in $\mathbb{R}^{n}$ with $n\geq2$. For $T>0$, denote $\Omega_{T}:=\Omega\times(-T,0]$. In this paper, we focus on a class of anisotropic weighted quasilinear parabolic equations as follows:
\begin{align}\label{PO001}
\begin{cases}
w_{1}\partial_{t}u-\mathrm{div}(w_{2}\mathbf{a}(x,t,u,\nabla u))=w_{2}b(x,t,u,\nabla u),& \mathrm{in}\;\Omega_{T},\\
\|u\|_{L^{\infty}(\Omega_{T})}\leq \mathcal{M}<\infty,
\end{cases}
\end{align}
where $w_{1}=|x'|^{\theta_{1}}|x|^{\theta_{2}}$, $w_{2}=|x'|^{\theta_{3}}|x|^{\theta_{4}}$, $x'=(x_{1},...,x_{n-1})$, the ranges of $\theta_{i}$, $i=1,2,3,4$ are prescribed in the following theorems, the functions $\mathbf{a}:\Omega_{T}\times\mathbb{R}^{n+1}\rightarrow\mathbb{R}^{n}$ and $b:\Omega_{T}\times\mathbb{R}^{n+1}\rightarrow\mathbb{R}$ are only measurable and subject to the following structure conditions: for $p>1$ and $(x.t)\in\Omega_{T}$,
\begin{itemize}
{\it
\item[$(\mathbf{H1})$] $\mathbf{a}(x,t,u,\nabla u)\cdot\nabla u\geq\lambda_{1}|\nabla u|^{p}-\phi_{1}(x,t)$,
\item[$(\mathbf{H2})$] $|\mathbf{a}(x,t,u,\nabla u)|\leq\lambda_{2}|\nabla u|^{p-1}+\phi_{2}(x,t)$,
\item[$(\mathbf{H3})$] $|b(x,t,u,\nabla u)|\leq\lambda_{3}|\nabla u|^{p-1}+\phi_{3}(x,t)$.}
\end{itemize}
Here $\lambda_{i}$, $i=1,2,3$ are given positive constants and $\phi_{i}$, $i=1,2,3$ are nonnegative functions such that
\begin{align}\label{ZE90}
\|\phi\|_{L^{l_{0}}(\Omega_{T},w_{2})}:=\left(\int_{\Omega_{T}}|\phi|^{l_{0}}w_{2}dxdt\right)^{\frac{1}{l_{0}}}<\infty,\quad \phi:=\phi_{1}+\phi_{2}^{\frac{p}{p-1}}+\phi_{3}^{\frac{p}{p-1}},
\end{align}
where $l_{0}$ satisfies
\begin{align}\label{E01}
l_{0}>\frac{n+p+\theta_{1}+\theta_{2}}{p}.
\end{align}
Remark that the anisotropy of the weight $|x'|^{\theta_{1}}|x|^{\theta_{2}}$ comes from $|x'|^{\theta_{1}}$. From the geometric point of view, $|x'|^{\theta_{1}}$ represents a degenerate or singular line, while $|x|^{\theta_{2}}$ only exhibits singularity or degeneracy at the origin. In \cite{LY202100} this type of weights were utilized to classify the singularities of all $(-1)$-homogeneous axisymmetric no-swirl solutions to the stationary Navier-Stokes equations found in \cite{LLY201801,LLY201802}. With regard to their properties and more applications in weighted Sobolev spaces and relevant PDEs, see e.g. \cite{LY2023,MZ2023,MZ202302}.

The structure conditions in $\mathrm{(}\mathbf{H1}\mathrm{)}$--$\mathrm{(}\mathbf{H3}\mathrm{)}$ were first proposed and studied in the well-known work \cite{D1986} completed by DiBenedetto, where the method of intrinsic scaling was developed to establish the H\"{o}lder estimates of solutions to degenerate parabolic equations. We here would like to point out that the assumed conditions in \eqref{ZE90}--\eqref{E01} imposed on $\phi_{i}$, $i=1,2,3$ are different from that in \cite{D1986} and have obvious advantages in simplifying the computations and presenting the following proofs in a more concise manner. Given these structure conditions, the degeneracy and singularity of \eqref{PO001} are of the same nature to the prototype as follows:
\begin{align*}
w_{1}\partial_{t}u-\mathrm{div}(w_{2}|\nabla u|^{p-2}\nabla u)=0.
\end{align*}
When $w_{1}=w_{2}=1$, it is the classical parabolic $p$-Laplace equation, which is frequently used to describe a large variety of diffusion phenomena occurring in natural sciences and engineering applications such as nonlinear porous medium flows and chemical concentration. For more relevant physical models and explanations, see \cite{BDGLS2023,DK2007,V2007} and the references therein. On one hand, according to the diffused feature of the flows, the equation can be divided into three types as follows. If $1<p<2$, it is termed fast diffusion equation and its solution will undergo extinction in finite time. The equation in the case of $p>2$, by contrast, is called slow diffusion equation and its solution always decays in the form of power-function to the stable state. The borderline case of $p=2$ corresponds to heat equation and its solution decays exponentially in the time variable. On the other hand, from the view of analysis, if $p>2$, the equation is said to be degenerate since its modulus of ellipticity $|\nabla u|^{p-2}$ degenerates to be zero at points where $|\nabla u|=0$, while if $1<p<2$, the equation is singular because $|\nabla u|^{p-2}$ blows up at points where $|\nabla u|=0$. In contrast to the case of $p=2$, these singular or degenerate nature in the case of $p\neq2$ weaken the smoothness of the solution to be of $C^{1,\alpha}$ for some $0<\alpha<1$. With regard to the regularity for quasilinear elliptic and parabolic equations without weights, the literature is very wide, see e.g. \cite{L2019,D1986,CD1988,D1983,D1993,U2008,E1982,DK1992,L1994,DGV2012,DGV2008,DGV200802,AL1983,DF198501,DF198502,DF1984,C1991,BSS2022,BDGLS2023} and the references therein.

For the weighted case, the situation becomes more complex. From the perspective of physical phenomena, the weights play a role in enhancing or reducing the diffusion rate of the flows. This fact can be observed by using the standard separation of variables method to obtain exact solution for heat equation with monomial weight $|x|^{\theta_{2}}$. The problem of studying their enhancement and weakening effects on the diffusion may possess a potential application in manufacturing porous medium materials with special permeation rates and dominating diffusion processes according to the requirements of the industry. From the angle of the structure of the above weighted equations, the modulus of ellipticity consisting of $|\nabla u|^{p-2}$ and the weight $w_{2}$ exhibits more complicated singular and degenerate behavior. This leads to that the solution will become worse and its regularity may further fall to be only of $C^{\alpha}$. In fact, the damage effect induced by the weights is stronger than that of $|\nabla u|^{p-2}$. Even when $p=2$, the smoothness of the solution has always been reduced to be of $C^{\alpha}$ due to their damage effect, except for some solutions of special structures such as even solutions found in \cite{STV2021}. The study on the weighted equations can date back to the famous work \cite{FKS1982}, where Fabes, Kenig and Serapioni \cite{FKS1982} proved the local H\"{o}lder continuity of solutions to second-order elliptic equations of divergence form with the weight $|x|^{\theta_{2}}$ for any $\theta_{2}>-n$. The recent work \cite{DLY2021,DLY2022} utilized spherical harmonic expansion to solve the sharp decay rate exponents of solutions to the weighted elliptic equations near the degenerate point of the weight, which clearly reflects the weakening effect of the weight on the regularity, see Lemma 2.2 in \cite{DLY2021} and Lemmas 2.2 and 5.1 in \cite{DLY2022} for more details. More importantly, Dong, Li and Yang \cite{DLY2021,DLY2022} used these exponents to find the optimal gradient blow-up rates for the insulated conductivity problem arising from composite materials, which has been previously regarded as a challenging problem. Recently, Miao and Zhao \cite{MZ202302} studied a class of anisotropic Muckenhoupt weights having more general form of $|x'|^{\theta_{1}}|x|^{\theta_{2}}|x_{n}|^{\theta_{3}}$ and raised a couple of intriguing problems involving anisotropic weighted interpolation and Poincar\'{e} inequalities and the classification for their enhancement and weakening effects on the diffusion. One of the main motivations in \cite{MZ202302} originates from the weight $|x_{n}|^{\theta_{3}}$, since this weight plays a significant role in the establishment of the global regularity for fast diffusion equations in \cite{JX2019,JX2022} completed by Jin and Xiong. Their results especially answered an open question raised by Berryman and Holland \cite{BH1980}. Subsequently, Jin, Ros-Oton and Xiong \cite{JRX2023} further extended the results to porous medium equations. For more investigations related to weighted elliptic and parabolic equations, we refer to \cite{FP2013,DP2023,DPT2023,DPT202302,MZ2023,JX2023,STV2021,STV202102,GW1991,GW1990,BS2023,S2010} and the references therein.

Before giving the definition of weak solution to problem \eqref{PO001}, we first list the required weighted spaces. Given a weight $w$, we use $L^{p}(\Omega,w)$, $L^{p}(\Omega_{T},w)$ and $W^{1,p}(\Omega,w)$ to denote the weighted $L^{p}$ spaces and weighted Sobolev spaces with their norms, respectively, given by
\begin{align*}
\begin{cases}
\|u\|_{L^{p}(\Omega,w)}=\left(\int_{\Omega}|u|^{p}wdx\right)^{\frac{1}{p}},\quad\|u\|_{L^{p}(\Omega_{T},w)}=\big(\int_{\Omega_{T}}|u|^{p}wdxdt\big)^{\frac{1}{p}},\vspace{0.3ex}\notag\\
\|u\|_{W^{1,p}(\Omega,w)}=\left(\int_{\Omega}|u|^{p}wdx\right)^{\frac{1}{p}}+\left(\int_{\Omega}|\nabla u|^{p}wdx\right)^{\frac{1}{p}}.
\end{cases}
\end{align*}
A function $u\in C((-T,0];L^{2}(\Omega,w_{1}))\cap L^{p}((-T,0);W^{1,p}(\Omega,w_{2}))$ is called a weak solution of \eqref{PO001}, provided that for any $-T<t_{1}<t_{2}\leq0$,
\begin{align*}
&\int_{\Omega}u\varphi w_{1}dx\Big|_{t_{1}}^{t_{2}}+\int_{t_{1}}^{t_{2}}\int_{\Omega}(-u\partial_{t}\varphi w_{1}+w_{2}\mathbf{a}(x,t,u,\nabla u)\cdot\nabla\varphi)dxdt\notag\\
&=\int_{t_{1}}^{t_{2}}\int_{\Omega}b(x,t,u,\nabla u)\varphi w_{2}dxdt,
\end{align*}
for all $\varphi\in W^{1,2}((0,T);L^{2}(\Omega,w_{1}))\cap L^{p}((0,T);W^{1,p}_{0}(\Omega,w_{2}))$.

Define
\begin{align*}
\begin{cases}
\mathcal{A}=\{(\theta_{1},\theta_{2}): \theta_{1}>-(n-1),\,\theta_{2}\geq0\},\\
\mathcal{B}=\{(\theta_{1},\theta_{2}):\theta_{1}>-(n-1),\,\theta_{2}<0,\,\theta_{1}+\theta_{2}>-n\},\\
\mathcal{C}_{p}=\{(\theta_{1},\theta_{2}):\theta_{1}<(n-1)(p-1),\,\theta_{2}\leq0\},\\
\mathcal{D}_{p}=\{(\theta_{1},\theta_{2}):\theta_{1}<(n-1)(p-1),\,\theta_{2}>0,\,\theta_{1}+\theta_{2}<n(p-1)\}.
\end{cases}
\end{align*}
Introduce the following exponent conditions:
\begin{itemize}
{\it
\item[$(\mathbf{K1})$] $(\theta_{1},\theta_{2})\in[(\mathcal{A}\cup\mathcal{B})\cap(C_{p}\cup\mathcal{D}_{p})]\cup\{\theta_{1}=0,\,\theta_{2}\geq n(p-1)\}$,\;$(\theta_{3},\theta_{4})\in\mathcal{A}\cup\mathcal{B}$,
\item[$(\mathbf{K2})$] $\theta_{1}+\theta_{2}>p-n,\;\theta_{1}\geq\theta_{3},\;\theta_{1}+\theta_{2}\geq\theta_{3}+\theta_{4},\;\theta_{3}+\min\{0,\theta_{4}\}>1-n$.}
\end{itemize}
We here give some explanations for the implications of $\mathrm{(}\mathbf{K1}\mathrm{)}$--$\mathrm{(}\mathbf{K2}\mathrm{)}$. First, $\mathcal{A}\cup\mathcal{B}$ is called the measure condition, which makes $w_{i}dx$, $i=1,2$ become two Radon measures. Second, $(\mathcal{A}\cup\mathcal{B})\cap(C_{p}\cup\mathcal{D}_{p})$ is introduced to guarantee that the considered weight $w_{1}$ belongs to the Muckenhoupt class $A_{p}$ (see \cite{M1972}) and we then obtain anisotropic weighted Poincar\'{e} inequality, which is critical to the establishment of the isoperimetric inequality of De Giorgi type. By contrast, the range of $\{\theta_{1}=0,\,\theta_{2}\geq n(p-1)\}$ is added by using the theories of quasiconformal mappings for the same purpose. See Section 2 in \cite{MZ2023} and Corollary 15.35 in \cite{HKM2006} for these statements. As for $\mathrm{(}\mathbf{K2}\mathrm{)}$, the range of $\theta_{1}+\theta_{2}>p-n$ is required to establish anisotropic weighted parabolic Sobolev inequality in Proposition \ref{prop001} below, while other ranges are used to ensure the validity of the switch from the measures $w_{1}dxdt$ to $w_{2}dxdt$, see the proofs in Section \ref{SEC05}.

Throughout this paper, let $C$ be a universal constant depending only on the data including $n,p,l_{0},\mathcal{M},\|\phi\|_{L^{l_{0}}(\Omega_{T},w_{2})}$, $\lambda_{i}$, $i=1,2,3$ and $\theta_{i}$, $i=1,2,3,4,$ whose value may change from line to line.

For later use, define two constants as follows:
\begin{align}\label{VAR01}
\varepsilon_{0}:=\frac{p(l_{0}-1)-n-\theta_{1}-\theta_{2}}{p(l_{0}-1)+2},\quad\vartheta:=\theta_{1}+\theta_{2}-\theta_{3}-\theta_{4}.
\end{align}
The first main result is concerned with the asymptotic behavior of solution to problem \eqref{PO001} near the singular or degenerate point of the weights.
\begin{theorem}\label{ZWTHM90}
Assume that $p>2$, $n\geq2$, $\mathrm{(}\mathbf{H1}\mathrm{)}$--$\mathrm{(}\mathbf{H3}\mathrm{)}$, \eqref{ZE90}--\eqref{E01} and $\mathrm{(}\mathbf{K1}\mathrm{)}$--$\mathrm{(}\mathbf{K2}\mathrm{)}$ hold. Let $u$ be a weak solution of problem \eqref{PO001} with $\Omega\times(-T,0]=B_{1}\times(-1,0]$. Then there exists a small constant $0<\alpha\leq\varepsilon_{0}$ depending only on the above data, such that for any fixed $t_{0}\in (-1/2,0]$ and all $(x,t)\in B_{1/2}\times(-1/2,t_{0}]$,
\begin{align}\label{QNAW001ZW}
u(x,t)=u(0,t_{0})+O(1)\big(|x|+|t-t_{0}|^{\frac{1}{p+\vartheta}}\big)^{\alpha},
\end{align}
where $\varepsilon_{0}$ and $\vartheta$ are defined by \eqref{VAR01}, $O(1)$ represents a quantity satisfying that $|O(1)|\leq C$ for some positive constant $C$ depending only on the data.
\end{theorem}

\begin{remark}
Fix the values of $n,p,l_{0}$. The upper bound exponent $\varepsilon_{0}$ tends to zero, as $\theta_{1}+\theta_{2}\rightarrow p(l_{0}-1)-n$. Moreover, we see from Remark \ref{W98} below that when $\varepsilon_{0}$ is a sufficiently small positive constant, the value of $\alpha$ in \eqref{QNAW001ZW} can attain the upper bound $\varepsilon_{0}$. In this case, if we let the value of $\theta_{1}+\theta_{2}$ increase towards to $p(l_{0}-1)-n$ and meanwhile decrease the value of $\theta_{3}+\theta_{4}$, then the regularity exponents $\varepsilon_{0}$ and $\frac{\varepsilon_{0}}{p+\vartheta}$ corresponding to the space and time variables all decrease, which shows the weakening effect of the weights on the regularity of $u$ and their reduction effect on the diffusion rate of the flows. An intriguing question naturally arises from these above facts. To be specific,
%\begin{itemize}
%{\it
%\item[] Fix the values of $n,p,l_{0},\theta_{3},\theta_{4}$. When the value of $\theta_{1}+\theta_{2}$ is sufficiently close to $p(l_{0}-1)-n$, whether $\varepsilon_{0}$ can become the optimal exponent in \eqref{QNAW001ZW}. If this is true, it proceeds to find the threshold value of $\varepsilon_{0}$.
%}
%\end{itemize}

\emph{\textbf{Question.} Fix the values of $n,p,l_{0},\theta_{3},\theta_{4}$. When the value of $\theta_{1}+\theta_{2}$ is sufficiently close to $p(l_{0}-1)-n$, whether $\varepsilon_{0}$ can become the optimal exponent in \eqref{QNAW001ZW}. If this is true, it proceeds to find the threshold value of $\varepsilon_{0}$.}

\end{remark}

When the weights $w_{i}$, $i=1,2$ are chosen to be the same type of single power-type weights and the considered equation becomes non-homogeneous weighted parabolic $p$-Laplace equation as follows:
\begin{align}\label{PO009}
\begin{cases}
w_{1}\partial_{t}u-\mathrm{div}(w_{2}|\nabla u|^{p-2}\nabla u)=w_{2}(\lambda_{3}|\nabla u|^{p-1}+\phi_{3}(x,t)),& \mathrm{in}\;\Omega_{T},\\
\|u\|_{L^{\infty}(\Omega_{T})}\leq \mathcal{M}<\infty,
\end{cases}
\end{align}
we further establish the H\"{o}lder estimates for the solution as follows.
\begin{theorem}\label{THM060}
Set $p>2$, $n\geq2$. Let \eqref{ZE90}--\eqref{E01} and $\mathrm{(}\mathbf{K1}\mathrm{)}$--$\mathrm{(}\mathbf{K2}\mathrm{)}$ hold. Suppose that $u$ is a weak solution of problem \eqref{PO009} with $\Omega\times(-T,0]=B_{1}\times(-1,0]$. Then there exists a small constant $0<\tilde{\alpha}<\frac{\alpha}{1+\alpha}$ with $\alpha$ determined by Theorem \ref{ZWTHM90}, such that if  $(w_{1},w_{2})=(|x'|^{\theta_{1}},|x'|^{\theta_{3}})$ or $(w_{1},w_{2})=(|x|^{\theta_{2}},|x|^{\theta_{4}})$,
\begin{align}\label{MWQ099}
|u(x,t)-u(y,s)|\leq C\big(|x-y|+|t-s|^{\frac{1}{p+\vartheta}}\big)^{\tilde{\alpha}},
\end{align}
for any $(x,t),(y,s)\in B_{1/4}\times(-1/2,0],$ where $\vartheta$ is given by \eqref{VAR01}.
\end{theorem}
\begin{remark}
According to the classical regularity theory for quasilinear parabolic equations, we know that the solution should be of $C^{1,\alpha}$ at the points away from the degenerate or singular points of the weights. This means that the regularity of solution to the considered weighted equation is actually determined by the behavior of the solution near these singular or degenerate points. Therefore, although we have provided a clear control relationship between the decay rate exponent $\alpha$ and the H\"{o}lder regularity exponent $\tilde{\alpha}$ in Theorem \ref{THM060}, it is worth expecting that the value of $\tilde{\alpha}$ in \eqref{MWQ099} can be further improved to be sufficiently close to the value of $\alpha$ in \eqref{QNAW001ZW}. The reason why we don't achieve this improvement is purely technical, and the problem on how to achieve the improvement remains open.

\end{remark}

\begin{remark}
For any given $R_{0}>0$, when we replace $B_{1}\times(-1,0]$ with $B_{R_{0}}\times(-R_{0}^{p+\vartheta},0]$ in Theorems \ref{ZWTHM90} and \ref{THM060}, we deduce from their proofs with minor modification that the results in \eqref{QNAW001ZW} and \eqref{MWQ099} hold with $(-R^{p+\vartheta}_{0}/2,0]$, $B_{R_{0}/2}\times(-R_{0}^{p+\vartheta}/2,t_{0}]$ and $B_{R_{0}/4}\times(-R^{p+\vartheta}_{0}/2,0]$ substituting for $(-1/2,0]$, $B_{1/2}\times(-1/2,t_{0}]$ and $B_{1/4}\times(-1/2,0]$, respectively. A difference is that the constant $C$ will depend upon $R_{0}$, but the exponents $\alpha$ and $\tilde{\alpha}$ not.

\end{remark}

\begin{remark}
Our results in Theorems \ref{ZWTHM90} and \ref{THM060} are novel and illuminating, which further develop and enrich the regularity theory of solutions to quasilinear parabolic equations.

\end{remark}

In order to complete the proofs of Theorems \ref{ZWTHM90} and \ref{THM060}, the key lies in making clear the behavior of the solution near the singular or degenerate point of the weights, which will be achieved by combining intrinsic scaling technique developed in \cite{D1986} and exponential variable substitution introduced in \cite{DGV2008}. The first step is to establish local energy estimates in Section \ref{SEC03}, which are the building blocks of the method of intrinsic scaling. Before that, we do some preliminary work in Section \ref{SEC002}. Subsequently, Theorems \ref{ZWTHM90} and \ref{THM060} are proved in Section \ref{SEC05}. The proofs rely on the establishments of three main ingredients required in the parabolic version of De Giorgi truncation method (see e.g. \cite{MZ2023,JX2022}), which consist of the expansion of time, the decay estimates and the oscillation improvement of the solution, see Lemmas \ref{LEM0035}, \ref{lem005} and \ref{LEM090} below for finer details. Remark that the idea of intrinsic scaling is to choose suitably rescaled cylinders whose dimensions accommodate the singularity and degeneracy of the equations and weights. By working with these cylinders, we recover the homogeneity between the space and time variables such that anisotropic weighted parabolic Sobolev inequality in Proposition \ref{prop001} can be applied to the improvement on oscillation of the solution in Lemma \ref{LEM090}. By comparison, the purpose of exponential variable substitution is to produce sufficiently large time interval for the transformed solution, which makes intrinsic scaling technique be successfully used to establish the desired decay estimates and improve the oscillation of the solution in a large cylinder. Finally, we give an alternative proof for the expansion of time by making use of the logarithmic estimates in the Appendix.

It is worth emphasizing that the anisotropic weights considered here will greatly increase the difficulties of analyses and computations due to their sophisticated forms. Especially it leads to distinct differences in the process of the technical implementation compared to the unweighted case in \cite{D1986,DGV2008}. In this paper, we optimize the proof procedures as much as possible by picking the concise conditions in \eqref{ZE90}--\eqref{E01} and presenting the proofs in the style resembling the classical De Giorgi truncation method of parabolic version. These improvements contribute to deepening the readers' understanding on intrinsic scaling technique and exponential variable substitution. More importantly, we capture an explicit upper bound $\varepsilon_{0}$ on the decay rate exponent $\alpha$ in Theorem \ref{ZWTHM90}, which can be attained under certain condition and meanwhile reveals the damage effect of the weights on the regularity of the solution and their weakening effect on the diffusion of the flows. Our results are new and allow for a lot of generalizations to more types of nonlinear weighted parabolic equations in future work.

\section{Preliminary}\label{SEC002}

We start by proving the following inequality.
\begin{lemma}\label{Lem865}
For any $a,b\geq0$ and $p>1$, we obtain that for any $\varepsilon>0$,

$(i)$ if $p>1$ is an integer,
\begin{align*}
(a+b)^{p}\leq(1+\varepsilon)a^{p}+\frac{C(p)}{\varepsilon^{p-1}}b^{p};
\end{align*}

$(ii)$ if $p>1$ is not an integer,
\begin{align*}
(a+b)^{p}\leq (1+\varepsilon)a^{p}+\frac{C(p)}{\varepsilon^{\frac{[p]}{p-[p]}}}b^{p},
\end{align*}
where $[p]$ denotes the integer part of $p$.
\end{lemma}
\begin{proof}
On one hand, if $p>1$ is an integer, we deduce from the binomial theorem and Young's inequality that
\begin{align}\label{MZ009}
(a+b)^{p}=&a^{p}+\sum^{p}_{j=1}C_{p}^{j}a^{p-j}b^{j}\leq(1+\varepsilon)a^{p}+C(p)b^{p}\sum^{p}_{j=1}\varepsilon^{-\frac{p-j}{j}}\notag\\ \leq&(1+\varepsilon)a^{p}+\frac{C(p)}{\varepsilon^{p-1}}b^{p},\quad C_{p}^{j}=\frac{p!}{j!(p-j)!},
\end{align}
where $j!$ represents the factorial of $j$.

On the other hand, if $p>1$ is not an integer, using \eqref{MZ009}, we obtain
\begin{align}\label{MZ013}
(a+b)^{p}=&(a+b)^{[p]+(p-[p])}\leq\left((1+\varepsilon)a^{[p]}+\frac{C(p)}{\varepsilon^{[p-1]}}b^{[p]}\right)\left(a^{p-[p]}+b^{p-[p]}\right)\notag\\
=&(1+\varepsilon)a^{p}+(1+\varepsilon)a^{[p]}b^{p-[p]}+\frac{C(p)}{\varepsilon^{[p-1]}}a^{p-[p]}b^{[p]}+\frac{C(p)}{\varepsilon^{[p-1]}}b^{p-2}.
\end{align}
A consequence of Young's inequality gives that
\begin{align*}
&(1+\varepsilon)a^{[p]}b^{p-[p]}\leq\varepsilon a^{p}+\frac{C(p)}{\varepsilon^{\frac{[p]}{p-[p]}}}b^{p},\quad\frac{C(p)}{\varepsilon^{[p-1]}}a^{p-[p]}b^{[p]}\leq\varepsilon a^{p}+\frac{C(p)}{\varepsilon^{p-1}}b^{p}.
\end{align*}
Substituting the above inequalities into \eqref{MZ013}, we have
\begin{align*}
(a+b)^{p}\leq(1+3\varepsilon)a^{p}+\frac{C(p)}{\varepsilon^{\frac{[p]}{p-[p]}}}b^{p}.
\end{align*}
The proof is complete.

\end{proof}

For $r,p,q>1$, let $f\in C((-1,0];L^{p}(B_{1},w_{1}))\cap L^{r}((-1,0);L^{q}(B_{1},w_{2}))$. Introduce the Steklov average as follows: for $0<h\ll1$,
\begin{align*}
f_{h}(x,t)=
\begin{cases}
\frac{1}{h}\int^{t}_{t-h}f(x,s)ds,& t\in(-1+h,0),\\
0,&t\leq-1+h,
\end{cases}
\end{align*}
and
\begin{align*}
f_{\bar{h}}(x,t)=
\begin{cases}
\frac{1}{h}\int^{t+h}_{t}f(x,s)ds,& t\in(-1,-h),\\
0,&t\geq-h.
\end{cases}
\end{align*}
\begin{lemma}\label{LEMA90}
Assume that $(\theta_{1},\theta_{2}),(\theta_{3},\theta_{4})\in\mathcal{A}\cup\mathcal{B}$. For $r,p,q>1$, if
$$f\in C((-1,0];L^{p}(B_{1},w_{1}))\cap L^{r}((-1,0);L^{q}(B_{1},w_{2})),$$
then for any $\delta\in(0,1)$,
\begin{align*}
\sup\limits_{t\in[-1+\delta,0]}\|(f_{h}-f)(\cdot,t)\|_{L^{p}(B_{1},w_{1})}\rightarrow0,\quad\text{as }h\rightarrow0,
\end{align*}
and
\begin{align*}
\left(\int^{0}_{-1+\delta}\|(f_{h}-f)(\cdot,t)\|_{L^{q}(B_{1},w_{2})}^{r}dt\right)^{\frac{1}{r}}\rightarrow0,\quad\text{as }h\rightarrow0.
\end{align*}

\end{lemma}
\begin{proof}
From Minkowski's inequality, we obtain that as $h\rightarrow0$,
\begin{align*}
\|(f_{h}-f)(\cdot,t)\|_{L^{p}(B_{1},w_{1})}\leq&\frac{1}{h}\int^{t}_{t-h}\|f(\cdot,s)-f(\cdot,t)\|_{L^{p}(B_{1},w_{1})}ds\notag\\
\leq&\sup\limits_{t-h\leq s\leq t}\|f(\cdot,s)-f(\cdot,t)\|_{L^{p}(B_{1},w_{1})}\rightarrow0.
\end{align*}
Similarly, using Minkowski's inequality twice, we have
\begin{align*}
&\left(\int^{0}_{-1+\delta}\|(f_{h}-f)(\cdot,t)\|_{L^{q}(B_{1},w_{2})}^{r}dt\right)^{\frac{1}{r}}\notag\\
&\leq\frac{1}{h}\int_{-h}^{0}\left(\int^{0}_{-1+\delta}\|f(\cdot,t+s)-f(\cdot,t)\|_{L^{q}(B_{1},w_{2})}^{r}dt\right)^{\frac{1}{r}}ds\notag\\
&\leq\sup\limits_{-h\leq s\leq0}\left(\int^{0}_{-1+\delta}\|f(\cdot,t+s)-f(\cdot,t)\|_{L^{q}(B_{1},w_{2})}^{r}dt\right)^{\frac{1}{r}}\rightarrow0,\quad\text{as }h\rightarrow0,
\end{align*}
where in the last inequality we utilized the continuity of Lebesgue integrals with respect to translations.

\end{proof}

Remark that Lemmas \ref{Lem865} and \ref{LEMA90} will be used in the process of establishing local energy estimates in Section \ref{SEC03}. We next state anisotropic weighted isoperimetric inequality and parabolic Sobolev embedding theorem.

\begin{lemma}\label{prop002}
For $n\geq2$ and $1<p<\infty$, let $(\theta_{1},\theta_{2})\in[(\mathcal{A}\cup\mathcal{B})\cap(C_{p}\cup\mathcal{D}_{p})]\cup\{\theta_{1}=0,\,\theta_{2}\geq n(p-1)\}$. Then there exists some constant $1<p_{\ast}=p_{\ast}(n,p,\theta_{1},\theta_{2})<p$ such that for any $R>0$, $l>k$ and $u\in W^{1,p_{\ast}}(B_{R},w_{1})$,
\begin{align}\label{pro001}
&(l-k)^{p_{\ast}}\bigg(\int_{\{u\geq l\}\cap B_{R}}w_{1}dx\bigg)^{p_{\ast}}\int_{\{u\leq k\}\cap B_{R}}w_{1}dx\notag\\
&\leq C(n,p,\theta_{1},\theta_{2})R^{p_{\ast}(n+\theta_{1}+\theta_{2}+1)}\int_{\{k<u<l\}\cap B_{R}}|\nabla u|^{p_{\ast}}w_{1}dx,
\end{align}
and
\begin{align*}
&(l-k)^{p_{\ast}}\bigg(\int_{\{u\leq k\}\cap B_{R}}w_{1}dx\bigg)^{p_{\ast}}\int_{\{u\geq l\}\cap B_{R}}w_{1}dx\notag\\
&\leq C(n,p,\theta_{1},\theta_{2})R^{p_{\ast}(n+\theta_{1}+\theta_{2}+1)}\int_{\{k<u<l\}\cap B_{R}}|\nabla u|^{p_{\ast}}w_{1}dx,
\end{align*}
where $w_{1}=|x'|^{\theta_{1}}|x|^{\theta_{2}}dx.$
\end{lemma}

\begin{proof}
A direct consequence of Proposition 2.10 in \cite{MZ2023} and Theorem 15.13 in \cite{HKM2006} shows that Lemma \ref{prop002} holds.

\end{proof}

Let $n\geq2$, $p>1$, $R>0$, and $-1\leq t_{1}<t_{2}\leq0$. For $u\in C((t_{1},t_{2});L^{p}(B_{R},w_{1}))\cap L^{p}((t_{1},t_{2});W_{0}^{1,p}(B_{R},w_{2}))$, write
\begin{align*}
\|u\|_{V^{p}_{0}(B_{R}\times(t_{1},t_{2}),w_{1},w_{2})}=\bigg(\sup\limits_{t\in(t_{1},t_{2})}\int_{B_{R}}|u|^{p}w_{1}dx+\int^{t_{2}}_{t_{1}}\int_{B_{R}}|\nabla u|^{p}w_{2}dxdt\bigg)^{\frac{1}{p}},
\end{align*}
where $w_{1}$ and $w_{2}$ are given in \eqref{PO001}. We now state anisotropic weighted parabolic Sobolev inequality as follows.
\begin{prop}\label{prop001}
For $n\geq2$, $p>1$, $R>0$, $\theta_{1}+\theta_{2}>p-n$, and $-1\leq t_{1}<t_{2}\leq0$, let $u\in C((t_{1},t_{2});L^{p}(B_{R},w_{1}))\cap L^{p}((t_{1},t_{2});W_{0}^{1,p}(B_{R},w_{2}))$. Then
\begin{align*}
\|u\|_{L^{p\chi}(B_{R}\times(t_{1},t_{2}),w_{2})}\leq C(n,p,\theta_{1},\theta_{2})\|u\|_{V^{p}_{0}(B_{R}\times(t_{1},t_{2}),w_{1},w_{2})},
\end{align*}
where $\chi=\frac{n+p+\theta_{1}+\theta_{2}}{n+\theta_{1}+\theta_{2}}.$

\end{prop}
\begin{proof}
Observe from the anisotropic Caffarelli-Kohn-Nirenberg inequality in \cite{LY2023} that for any $u\in W_{0}^{1,p}(B_{R},w_{2})$,
\begin{align*}
&\left(\int_{B_{R}}|u|^{\frac{p(n+\theta_{1}+\theta_{2})}{n+\theta_{1}+\theta_{2}-p}}|x'|^{\frac{\theta_{3}(n+\theta_{1}+\theta_{2})-p\theta_{1}}{n+\theta_{1}+\theta_{2}-p}}|x|^{\frac{\theta_{4}(n+\theta_{1}+\theta_{2})-p\theta_{2}}{n+\theta_{1}+\theta_{2}-p}}dx\right)^{\frac{n+\theta_{1}+\theta_{2}-p}{n+\theta_{1}+\theta_{2}}}\notag\\
&\leq C\int_{B_{R}}|\nabla u|^{p}|x'|^{\theta_{3}}|x|^{\theta_{4}}dx,
\end{align*}
which, together with H\"{o}lder's inequality, reads that
\begin{align}\label{ZW001}
&\int_{B_{R}}|u|^{p\chi}|x'|^{\theta_{3}}|x|^{\theta_{4}}dx\notag\\
&=\int_{B_{R}}|u|^{p}|x'|^{\theta_{3}-\theta_{1}(\chi-1)}|x|^{\theta_{4}-\theta_{2}(\chi-1)}|u|^{p(\chi-1)}|x'|^{\theta_{1}(\chi-1)}|x|^{\theta_{2}(\chi-1)}dx\notag\\
&\leq\left(\int_{B_{R}}|u|^{\frac{p}{2-\chi}}|x'|^{\frac{\theta_{3}-\theta_{1}(\chi-1)}{2-\chi}}|x|^{\frac{\theta_{4}-\theta_{2}(\chi-1)}{2-\chi}}dx\right)^{2-\chi}\left(\int_{B_{R}}|u|^{p}|x'|^{\theta_{1}}|x|^{\theta_{2}}dx\right)^{\chi-1}\notag\\
&\leq C\int_{B_{R}}|\nabla u|^{p}|x'|^{\theta_{3}}|x|^{\theta_{4}}dx\left(\int_{B_{R}}|u|^{p}|x'|^{\theta_{1}}|x|^{\theta_{2}}dx\right)^{\chi-1}.
\end{align}
Integrating \eqref{ZW001} from $t_{1}$ to $t_{2}$ and using Young's inequality, we obtain
\begin{align*}
\left(\int^{t_{2}}_{t_{1}}\int_{B_{R}}|u|^{p\chi}w_{2}\right)^{\frac{1}{\chi}}\leq& C\bigg(\sup\limits_{t\in(t_{1},t_{2})}\int_{B_{R}}|u|^{p}w_{1}dx\bigg)^{\frac{\chi-1}{\chi}}\left(\int_{B_{R}}|\nabla u|^{p}w_{2}dx\right)^{\frac{1}{\chi}}\notag\\
\leq&C\bigg(\sup\limits_{t\in(t_{1},t_{2})}\int_{B_{R}}|u|^{p}w_{1}dx+\int^{t_{2}}_{t_{1}}\int_{B_{R}}|\nabla u|^{p}w_{2}dxdt\bigg).
\end{align*}
The proof is finished.

\end{proof}

\section{Local energy estimates}\label{SEC03}

For later use, we first fix some notations. For $x_{0}\in\mathbb{R}^{n}$, $t_{0}\in\mathbb{R}$ and $\rho,\tau>0$, let $B_{\rho}(x_{0})$ be the ball of centre $x_{0}$ and radius $\rho$. We use $[(x_{0},t_{0})+Q(\rho,\tau)]:=B_{\rho}(x_{0})\times(t_{0}-\tau,t_{0})$ to denote the backward cylinder of radius $\rho$ and height $\tau$ with vertex at $(x_{0},t_{0})$. For brevity, write $B_{\rho}:=B_{\rho}(0)$ and $Q(\rho,\tau):=[(0,0)+Q(\rho,\tau)]$. For $k\in\mathbb{R}$ and $u\in C((-1,0];L^{2}(B_{1},w_{1}))\cap L^{p}((-1,0);W^{1,p}(B_{1},w_{2}))$, define
\begin{align*}
(u-k)_{+}=\max\{u-k,0\},\quad(u-k)_{-}=\max\{k-u,0\}.
\end{align*}
For $E\subset B_{1}$ and $\widetilde{E}\subset B_{1}\times(-1,0]$, let
\begin{align}\label{MEA01}
|E|_{\mu_{w_{i}}}=\int_{E}w_{i}dx,\quad |\widetilde{E}|_{\nu_{w_{i}}}=\int_{\widetilde{E}}w_{i}dxdt,\quad i=1,2.
\end{align}
Local energy estimates for the truncated solution are now listed as follows.
\begin{lemma}\label{lem003}
Suppose that $u$ is the solution to problem \eqref{PO001} with $\Omega_{T}=B_{1}\times(-1,0]$. Let $v_{\pm}:=(u-k)_{\pm}$ with $k\in\mathbb{R}$. For any $[(x_{0},t_{0})+Q(\rho,\tau)]\subset B_{1}\times(-1,0]$ and $\xi\in C^{\infty}([(x_{0},t_{0})+Q(\rho,\tau)])$ which vanishes on $\partial B_{\rho}(x_{0})\times(t_{0}-\tau,t_{0})$ and satisfies that $0\leq\xi\leq1$, we derive
\begin{align*}
&\sup\limits_{s\in(t_{0}-\tau,t_{0})}\bigg\{\int_{B_{\rho}(x_{0})}v_{\pm}^{2}\xi^{p}(x,s)w_{1}dx+\frac{\lambda_{1}}{3}\int_{B_{\rho}(x_{0})\times(t_{0}-\tau,s)}|\nabla(v_{\pm}\xi)|^{p}w_{2}dxdt\bigg\}\notag\\
&\leq\int_{B_{\rho}(x_{0})}v_{\pm}^{2}\xi^{p}(x,t_{0}-\tau)w_{1}dx\notag\\
&\quad+C\int_{[(x_{0},t_{0})+Q(\rho,\tau)]}(v_{\pm}^{2}\xi^{p-1}|\partial_{t}\xi|w_{1}+v_{\pm}^{p}(|\nabla\xi|^{p}+|\xi|^{p})w_{2})dxdt\notag\\
&\quad+C\|\phi\|_{L^{l_{0}}(B_{1}\times(-1,0),w_{2})}|[(x_{0},t_{0})+Q(\rho,\tau)]\cap \{v_{\pm}>0\}|_{\nu_{w_{2}}}^{1-\frac{1}{l_{0}}},
\end{align*}
where $\phi=\phi_{1}+\phi_{2}^{\frac{p}{p-1}}+\phi_{3}^{\frac{p}{p-1}}$.
\end{lemma}

\begin{proof}
Without loss of generality, let $(x_{0},t_{0})=(0,0)$. Set $\varphi=\pm(u_{h}-k)_{\pm}\xi^{p}$. Then picking the test function $\varphi_{\bar{h}}$, we obtain that for any $-\tau\leq s\leq 0$,
\begin{align}\label{WM01}
&\int_{B_{\rho}}u\varphi_{\bar{h}} w_{1}dx\Big|_{-\tau}^{s}+\int_{-\tau}^{s}\int_{B_{\rho}}(-u\partial_{t}\varphi_{\bar{h}} w_{1}+w_{2}\mathbf{a}(x,t,u,\nabla u)\cdot\nabla\varphi_{\bar{h}})dxdt\notag\\
&=\int_{-\tau}^{s}\int_{B_{\rho}}b(x,t,u,\nabla u)\varphi_{\bar{h}} w_{2}dxdt.
\end{align}
By a direct calculation, we deduce
\begin{align*}
&\int_{-\tau}^{s}\int_{B_{\rho}}(\partial_{t}u_{h}\varphi w_{1}+w_{2}[\mathbf{a}(x,t,u,\nabla u)]_{h}\cdot\nabla\varphi)dxdt\notag\\
&=\int_{-\tau}^{s}\int_{B_{\rho}}[b(x,t,u,\nabla u)]_{h}\varphi w_{2}dxdt.
\end{align*}
First, we have from integration by parts and Lemma \ref{LEMA90} that
\begin{align*}
&\int_{-\tau}^{s}\int_{B_{\rho}}\partial_{t}u_{h}\varphi w_{1}dxdt=\int_{-\tau}^{s}\int_{B_{\rho}}\partial_{t}[(u_{h}-k)_{\pm}]^{2}\xi^{p}w_{1}dxdt\notag\\
&=\frac{1}{2}\int_{B_{\rho}}((u_{h}-k)_{\pm}^{2}\xi^{p}(x,s)-(u_{h}-k)_{\pm}^{2}\xi^{p}(x,-\tau))w_{1}dx\notag\\
&\quad-\frac{p}{2}\int_{B_{\rho}\times(-\tau,s)}(u_{h}-k)_{\pm}^{2}\xi^{p-1}\partial_{t}\xi w_{1}dxdt\notag\\
&\rightarrow\frac{1}{2}\int_{B_{\rho}}((u-k)_{\pm}^{2}\xi^{p}(x,s)-(u-k)_{\pm}^{2}\xi^{p}(x,-\tau))w_{1}dx\notag\\
&\quad-\frac{p}{2}\int_{B_{\rho}\times(-\tau,s)}(u-k)_{\pm}^{2}\xi^{p-1}\partial_{t}\xi w_{1}dxdt,\quad\text{as }h\rightarrow0.
\end{align*}
With regard to the remaining parts in \eqref{WM01}, by first sending $h\rightarrow0$ and then making use of the structure conditions in $\mathrm{(}\mathbf{H1}\mathrm{)}$--$\mathrm{(}\mathbf{H3}\mathrm{)}$, we obtain from Young's inequality and Lemma \ref{LEMA90} that
\begin{align*}
&\int_{-\tau}^{s}\int_{B_{\rho}}w_{2}[\mathbf{a}(x,t,u,\nabla u)]_{h}\cdot\nabla\varphi dxdt\notag\\
&\rightarrow\int_{-\tau}^{s}\int_{B_{\rho}}\mathbf{a}(x,t,u,\nabla u)\cdot\big[\pm\nabla(u-k)_{\pm}\xi^{p}\pm p(u-k)_{\pm}\xi^{p-1}\nabla\xi\big]w_{2}dxdt\notag\\
&\geq\lambda_{1}\int_{-\tau}^{s}\int_{B_{\rho}}|\nabla(u-k)_{\pm}|^{p}\xi^{p}w_{2}dxdt-\int_{(B_{\rho}\times(-\tau,s))\cap\{(u-k)_{\pm}>0\}}\phi_{1}\xi^{p}w_{2}dxdt\notag\\
&\quad-\lambda_{2}p\int_{B_{\rho}\times(-\tau,s)}|\nabla(u-k)_{\pm}|^{p-1}(u-k)_{\pm}\xi^{p-1}|\nabla\xi|w_{2}dxdt\notag\\
&\quad-p\int_{B_{\rho}\times(-\tau,s)}(u-k)_{\pm}\phi_{2}\xi^{p-1}|\nabla\xi|w_{2}dxdt\notag\\
&\geq\frac{5\lambda_{1}}{6}\int_{-\tau}^{s}\int_{B_{\rho}}|\nabla(u-k)_{\pm}|^{p}\xi^{p}w_{2}dxdt-C\int_{B_{\rho}\times(-\tau,s)}(u-k)_{\pm}^{p}|\nabla\xi|^{p}w_{2}dxdt\notag\\
&\quad-C\int_{(B_{\rho}\times(-\tau,s))\cap\{(u-k)_{\pm}>0\}}\big(\phi_{1}+\phi_{2}^{\frac{p}{p-1}}\big)w_{2}dxdt,
\end{align*}
and
\begin{align*}
&\int_{-\tau}^{s}\int_{B_{\rho}}[b(x,t,u,\nabla u)]_{h}\varphi w_{2}\rightarrow\int_{-\tau}^{s}\int_{B_{\rho}}\pm b(x,t,u,\nabla u)(u-k)_{\pm}\xi^{p}w_{2}\notag\\
&\leq\lambda_{3}\int_{B_{\rho}\times(-\tau,s)}|\nabla(u-k)_{\pm}|^{p-1}(u-k)_{\pm}\xi^{p}w_{2}+\int_{B_{\rho}\times(-\tau,s)}(u-k)_{\pm}\phi_{3}\xi^{p}w_{2}\notag\\
&\leq\frac{\lambda_{1}}{6}\int_{B_{\rho}\times(-\tau,s)}|\nabla(u-k)_{\pm}|^{p}\xi^{p}w_{2}+C\int_{B_{\rho}\times(-\tau,s)}(u-k)_{\pm}^{p}\xi^{p}w_{2}\notag\\
&\quad+\int_{(B_{\rho}\times(-\tau,s))\cap\{(u-k)_{\pm}>0\}}\phi_{3}^{\frac{p}{p-1}}w_{2}.
\end{align*}
Combining these above facts, we deduce from H\"{o}lder's inequality that for any $-\tau\leq s\leq 0$,
\begin{align}\label{QE965}
&\int_{B_{\rho}}v_{\pm}^{2}\xi^{p}(x,s)w_{1}dx+\frac{2\lambda_{1}}{3}\int_{B_{\rho}\times(-\tau,s)}|\xi\nabla v_{\pm}|^{p}w_{2}dxdt\notag\\
&\leq\int_{B_{\rho}}v_{\pm}^{2}\xi^{p}(x,-\tau)w_{1}dx+C\int_{Q(\rho,\tau)}(v_{\pm}^{2}\xi^{p-1}|\partial_{t}\xi|w_{1}+v_{\pm}^{p}(|\nabla\xi|^{p}+|\xi|^{p})w_{2})dxdt\notag\\
&\quad+C\int_{Q(\rho,\tau)\cap\{(u-k)_{\pm}>0\}}\big(\phi_{1}+\phi_{2}^{\frac{p}{p-1}}+\phi_{3}^{\frac{p}{p-1}}\big)w_{2}dxdt\notag\\
&\leq\int_{B_{\rho}}v_{\pm}^{2}\xi^{p}(x,-\tau)w_{1}dx+C\int_{Q(\rho,\tau)}(v_{\pm}^{2}\xi^{p-1}|\partial_{t}\xi|w_{1}+v_{\pm}^{p}(|\nabla\xi|^{p}+|\xi|^{p})w_{2})dxdt\notag\\
&\quad+C\|\phi\|_{L^{l_{0}}(B_{1}\times(-1,0),w_{2})}|Q(\rho,\tau)\cap \{v_{\pm}>0\}|_{\nu_{w_{2}}}^{1-\frac{1}{l_{0}}},
\end{align}
where $v_{\pm}=(u-k)_{\pm}$ and $\phi=\phi_{1}+\phi_{2}^{\frac{p}{p-1}}+\phi_{3}^{\frac{p}{p-1}}.$ Then applying Lemma \ref{Lem865} with $\varepsilon=1$ to \eqref{QE965}, we complete the proof of Lemma \ref{lem003}.

\end{proof}

%When $\lambda_{1}$ is sufficiently greater than $\lambda_{3}$, we can further improve the upper bound in Lemma \ref{lem003} as follows.
%\begin{corollary}
%Assume as in Lemma \ref{lem003}. If $\lambda_{1}\gg\lambda_{3}$, we have
%\begin{align*}
%&\sup\limits_{s\in(t_{0}-\tau,t_{0})}\bigg\{\int_{B_{\rho}(x_{0})}v_{\pm}^{2}\xi^{p}(x,s)w_{1}dx+\frac{\lambda_{1}}{6}\int_{B_{\rho}(x_{0})\times(t_{0}-\tau,s)}|\nabla(v_{\pm}\xi)|^{p}w_{2}dxdt\bigg\}\notag\\
%&\leq\int_{B_{\rho}(x_{0})}v_{\pm}^{2}\xi^{p}(x,t_{0}-\tau)w_{1}dx\notag\\
%&\quad+C\int_{[(x_{0},t_{0})+Q(\rho,\tau)]}(v_{\pm}^{2}\xi^{p-1}|\partial_{t}\xi|w_{1}+v_{\pm}^{p}|\nabla\xi|^{p}w_{2})dxdt\notag\\
%&\quad+C\|\phi\|_{L^{l_{0}}(B_{1}\times(-1,0),w_{2})}|[(x_{0},t_{0})+Q(\rho,\tau)]\cap \{v_{\pm}>0\}|_{\nu_{w_{2}}}^{1-\frac{1}{l_{0}}}.
%\end{align*}
%
%\end{corollary}
%\begin{proof}
%It suffices to improve the estimates in \eqref{Z09}. By using Theorem 15.23 in \cite{HKM2006}, we obtain the following weighted Poincar\'{e} inequality:
%\begin{align*}
%a
%\end{align*}
%
%\end{proof}

\section{Local regularity for weak solutions }\label{SEC05}

%Motivated by the idea developed in \cite{D1986}, in the following we find suitable cylinders whose dimensions accommodate the singularity and degeneracy of the weights and equations for the purpose of applying the De Giorgi truncation method to the considered weighted parabolic equations.

Denote
\begin{align*}
\bar{\varepsilon}_{0}:=(p-2)\varepsilon_{0},
\end{align*}
where $\varepsilon_{0}$ is given by \eqref{VAR01}. In light of $p>2$ and using the exponent conditions in \eqref{E01} and $\mathrm{(}\mathbf{K1}\mathrm{)}$--$\mathrm{(}\mathbf{K2}\mathrm{)}$, we obtain that for any $R\in(0,\frac{1}{2}]$ and $t_{0}\in[-\frac{1}{2},0]$, there holds
\begin{align*}
[(0,t_{0})+Q(2R,R^{p+\vartheta-\bar{\varepsilon}_{0}})]\subset B_{1}\times(-1,0].
\end{align*}
By a translation, in the following we assume without loss of generality that $(0,t_{0})=(0,0).$ Let
\begin{align*}
\mu^{+}=\sup\limits_{Q(2R,R^{p+\vartheta-\bar{\varepsilon}_{0}})}u,\quad\mu^{-}=\inf\limits_{Q(2R,R^{p+\vartheta-\bar{\varepsilon}_{0}})}u,
\end{align*}
and
\begin{align*}
\omega=\mathop{osc}\limits_{Q(2R,R^{p+\vartheta-\bar{\varepsilon}_{0}})}u=\mu^{+}-\mu^{-}.
\end{align*}
Construct the cylinder
\begin{align*}
Q(R,a_{0}R^{p+\vartheta}),\quad a_{0}=\left(\frac{\omega}{A}\right)^{2-p},
\end{align*}
where the constant $A$ will be determined later, which depends only upon the data. Observe that one of the following two relations must hold: either $\omega\leq AR^{\varepsilon_{0}}$, or $\omega>AR^{\varepsilon_{0}}$. In the case when $\omega>AR^{\varepsilon_{0}}$, we have
\begin{align*}
Q(R,a_{0}R^{p+\vartheta})\subset Q(2R,R^{p+\vartheta-\bar{\varepsilon}_{0}}),
\end{align*}
and thus,
\begin{align}\label{AQ821}
\mathop{osc}\limits_{Q(R,a_{0}R^{p+\vartheta})}u\leq\omega.
\end{align}
Remark that this fact is generally not verifiable for such a given box since its dimensions would have to be intrinsically matched with regard to the essential oscillation of the solution inside it. So we introduce the cylinder $Q(R,R^{p+\vartheta-\bar{\varepsilon}_{0}})$ to guarantee that \eqref{AQ821} holds under the condition of $\omega>AR^{\varepsilon_{0}}$. We further construct subcylinders of smaller size inside $Q(R,a_{0}R^{p+\vartheta})$ as follows:
\begin{align*}
Q(R,mR^{p+\vartheta}),\quad m=\left(\frac{\omega}{M}\right)^{2-p},
\end{align*}
where $0<M\leq A$. This implies that $Q(R,mR^{p+\vartheta})\subset Q(R,a_{0}R^{p+\vartheta}).$

The key to the proofs of Theorems \ref{ZWTHM90} and \ref{THM060} lies in achieving the following desired oscillation improvement of the solution $u$.
\begin{prop}\label{PRO01}
Assume as above. The constants $A,M$ (and then $a_{0},m$) can be determined and there exists two constants $\kappa_{\ast}>1$ and $1<c_{\ast}\leq A^{p-2}$, both depending only on the data such that we have either $\omega\leq AR^{\varepsilon_{0}}$, or

$(i)$ if
\begin{align}\label{A90}
|B_{R}\cap\{u(\cdot,-c_{\ast}\omega^{2-p}R^{p+\vartheta})>\mu^{+}-2^{-1}\omega\}|_{\mu_{w_{1}}}\leq2^{-1}|B_{R}|_{\mu_{w_{1}}},
\end{align}
then
\begin{align}\label{A01}
u(x,t)\leq\mu^{+}-2^{-\kappa_{\ast}}\omega,\quad\mathrm{for}\;(x,t)\in Q(R/2,m(R/2)^{p+\vartheta});
\end{align}

$(ii)$ if
\begin{align}\label{A91}
|B_{R}\cap\{u(\cdot,-c_{\ast}\omega^{2-p}R^{p+\vartheta})<\mu^{-}+2^{-1}\omega\}|_{\mu_{w_{1}}}\leq2^{-1}|B_{R}|_{\mu_{w_{1}}},
\end{align}
then
\begin{align}\label{A02}
u(x,t)\geq\mu^{-}+2^{-\kappa_{\ast}}\omega,\quad\mathrm{for}\;(x,t)\in Q(R/2,m(R/2)^{p+\vartheta}).
\end{align}

\end{prop}
\begin{remark}
Since $\frac{p+\vartheta}{p-2}>\varepsilon_{0}$ and $A\geq c_{\ast}^{\frac{1}{p-2}}$, then we have
\begin{align*}
c_{\ast}\omega^{2-p}R^{p+\vartheta}<1,\quad \text{if }\omega>AR^{\varepsilon_{0}}>c_{\ast}^{\frac{1}{p-2}}R^{\frac{p+\vartheta}{p-2}},
\end{align*}
which indicates that the assumed conditions in \eqref{A90} and \eqref{A91} are valid.
\end{remark}

\subsection{The proof of Proposition \ref{PRO01}}
In the following we take the proof of \eqref{A01} for example. The proof of \eqref{A02} is similar and thus omitted. We now begin with recalling a measure lemma as follows.
\begin{lemma}[see Lemma 2.1 in \cite{MZ2023}]\label{LEM860}
$d\mu:=|x'|^{\theta_{1}}|x|^{\theta_{2}}dx$ is a Radon measure if $(\theta_{1},\theta_{2})\in\mathcal{A}\cup\mathcal{B}$. Moreover, $C^{-1}R^{n+\theta_{1}+\theta_{2}}\leq\mu(B_{R})\leq C R^{n+\theta_{1}+\theta_{2}}$ for any $R>0$ and some constant $C=C(n,\theta_{1},\theta_{2})>0$.

\end{lemma}

Making use of Lemma \ref{LEM860}, we prove the following switch lemma from the measures $w_{1}dxdt$ to $w_{2}dxdt.$
\begin{lemma}\label{lem09}
Assume that $(\theta_{1},\theta_{2}),(\theta_{3},\theta_{4})\in\mathcal{A}\cup\mathcal{B}$ and $\theta_{3}+\min\{0,\theta_{4}\}>1-n$. There exists a positive constant $C_{0}=C_{0}(n,p,\theta_{1},\theta_{2},\theta_{3},\theta_{4})$ such that for any $\varepsilon,\rho\in(0,1/2]$ and $\widetilde{E}\subset Q(\rho,\rho^{p+\vartheta})$, if
\begin{align}\label{COND01}
\frac{|\widetilde{E}|_{\nu_{w_{1}}}}{|Q(\rho,\rho^{p+\vartheta})|_{\nu_{w_{1}}}}\leq\varepsilon^{\beta},
\end{align}
then
\begin{align*}
\frac{|\widetilde{E}|_{\nu_{w_{2}}}}{|Q(\rho,\rho^{p+\vartheta})|_{\nu_{w_{2}}}}\leq C_{0}\varepsilon^{n-1+\theta_{3}+\min\{0,\theta_{4}\}},
\end{align*}
where $\beta=n-1+\theta_{3}+\min\{0,\theta_{4}\}+\max\{0,\theta_{1}-\theta_{3}\}+\max\{0,\theta_{2}-\theta_{4}\}$, the measures $\nu_{w_{i}}$, $i=1,2$ are defined by \eqref{MEA01}.

\end{lemma}
\begin{proof}
Observe from Lemma \ref{LEM860} that
\begin{align*}
\frac{|\widetilde{E}|_{\nu_{w_{2}}}}{|Q(\rho,\rho^{p+\vartheta})|_{\nu_{w_{2}}}}\leq&\frac{C}{\rho^{n+p+\theta_{1}+\theta_{2}}}\bigg(\int_{\widetilde{E}\cap\{|x'|<\varepsilon \rho\}}w_{2}dxdt+\int_{\widetilde{E}\cap\{|x'|\geq\varepsilon \rho\}}w_{2}dxdt\bigg).
\end{align*}
For brevity, write
\begin{align*}
I_{1}:=\int_{\widetilde{E}\cap\{|x'|<\varepsilon \rho\}}w_{2}dxdt,\quad I_{2}:=\int_{\widetilde{E}\cap\{|x'|\geq\varepsilon \rho\}}w_{2}dxdt.
\end{align*}
For the first term $I_{1}$, we discuss as follows:

$(i)$ if $\theta_{4}\geq0$, then
\begin{align*}
I_{1}\leq C\rho^{\theta_{4}}\int_{\widetilde{E}\cap\{|x'|<\varepsilon \rho\}}|x'|^{\theta_{3}}dxdt\leq C\varepsilon^{n-1+\theta_{3}}\rho^{n+p+\theta_{1}+\theta_{2}};
\end{align*}

$(ii)$ if $\theta_{4}<0$, then
\begin{align*}
I_{1}\leq\int_{\widetilde{E}\cap\{|x'|<\varepsilon \rho\}}|x'|^{\theta_{3}+\theta_{4}}dxdt\leq C\varepsilon^{n-1+\theta_{3}+\theta_{4}}\rho^{n+p+\theta_{1}+\theta_{2}}.
\end{align*}

As for the second term $I_{2}$, we deduce from \eqref{COND01} that

$(i)$ if $\theta_{1}\geq\theta_{3}$, then
\begin{align*}
I_{2}\leq(\varepsilon \rho)^{\theta_{3}-\theta_{1}}\int_{\widetilde{E}\cap\{|x'|\geq\varepsilon \rho\}}|x'|^{\theta_{1}}|x|^{\theta_{4}}dxdt,
\end{align*}
which implies that for $\theta_{2}\geq\theta_{4}$,
\begin{align*}
I_{2}\leq(\varepsilon \rho)^{-\vartheta}|\widetilde{E}|_{\nu_{w_{1}}}\leq C\varepsilon^{\beta-\vartheta}\rho^{n+p+\theta_{1}+\theta_{2}},
\end{align*}
and, for $\theta_{2}<\theta_{4}$,
\begin{align*}
I_{2}\leq\varepsilon^{\theta_{3}-\theta_{1}}\rho^{-\vartheta}|\widetilde{E}|_{\nu_{w_{1}}}\leq C\varepsilon^{\beta-\theta_{1}+\theta_{3}}\rho^{n+p+\theta_{1}+\theta_{2}};
\end{align*}

$(ii)$ if $\theta_{1}<\theta_{3}$, then
\begin{align*}
I_{2}\leq\rho^{\theta_{3}-\theta_{1}}\int_{\widetilde{E}\cap\{|x'|\geq\varepsilon \rho\}}|x'|^{\theta_{1}}|x|^{\theta_{4}}dxdt,
\end{align*}
which leads to that for $\theta_{2}\geq\theta_{4}$,
\begin{align*}
I_{2}\leq\varepsilon^{\theta_{4}-\theta_{2}}\rho^{-\vartheta}|\widetilde{E}|_{\nu_{w_{1}}}\leq C\varepsilon^{\beta-\theta_{2}+\theta_{4}}\rho^{n+p+\theta_{1}+\theta_{2}},
\end{align*}
and, for $\theta_{2}<\theta_{4}$,
\begin{align*}
I_{2}\leq\rho^{-\vartheta}|\widetilde{E}|_{\nu_{w_{1}}}\leq C\varepsilon^{\beta}\rho^{n+p+\theta_{1}+\theta_{2}}.
\end{align*}
Therefore, combining these above facts, we have
\begin{align*}
\frac{|\widetilde{E}|_{\nu_{w_{2}}}}{|Q(\rho,\rho^{p+\vartheta})|_{\nu_{w_{2}}}}\leq&C\varepsilon^{n-1+\theta_{3}+\min\{0,\theta_{4}\}}.
\end{align*}

\end{proof}
In order to complete the proof of Proposition \ref{PRO01}, we need to establish the following three main lemmas required in the De Giorgi truncation method of parabolic version. The first lemma is to carry out the expansion of time by utilizing the energy estimates.
\begin{lemma}\label{LEM0035}
Assume as before. Then there exists a constant $\delta_{0}\in(0,1)$ depending only on the data and independent of $\omega$ such that we obtain either $\omega\leq R^{\varepsilon_{0}}$, or if for some $\bar{t}\in[-1,-\delta_{0}\omega^{2-p}R^{p+\vartheta}]$,
\begin{align}\label{Q01}
|B_{R}\cap\{u(\cdot,\bar{t})>\mu^{+}-2^{-1}\omega\}|_{\mu_{w_{1}}}\leq2^{-1}|B_{R}|_{\mu_{w_{1}}},
\end{align}
then
\begin{align*}
|B_{R}\cap\{u(\cdot,t)>\mu^{+}-2^{-c_{0}}\omega\}|_{\mu_{w_{1}}}\leq\frac{3}{4}|B_{R}|_{\mu_{w_{1}}},\quad\forall\,t\in[\bar{t},\bar{t}+\delta_{0}\omega^{2-p}R^{p+\vartheta}],
\end{align*}
where $c_{0}$ is given by
\begin{align}\label{CON09}
c_{0}:=-\frac{\ln\frac{\sqrt{5}-2}{2\sqrt{5}}}{\ln2}.
\end{align}

\end{lemma}

\begin{remark}
An alternative proof for the similar expansion of time, which involves application of the logarithmic estimates, is left in the Appendix.
\end{remark}

\begin{proof}
For $\delta_{0}\in(0,1]$ and $k\in[\mu^{-},\mu^{+}]$, write
\begin{align*}
A^{\delta_{0}}(k,R)=(B_{R}\times[\bar{t},\bar{t}+\delta_{0}\omega^{2-p} R^{p+\vartheta}])\cap\{u>k\}.
\end{align*}
Choose a smooth cut-off function $\zeta\in C^{\infty}_{0}(B_{R})$ satisfying that $\zeta=1$ in $B_{(1-\varrho)R}$, and $0\leq\zeta\leq1$, $|\nabla\zeta|\leq\frac{2}{\varrho R}$ in $B_{R}$, where $\varrho\in(0,1)$ is to be determined later. For brevity, denote $v=(u-(\mu^{+}-2^{-1}\omega))_{+}$. In light of \eqref{Q01} and $\delta_{0}\in(0,1]$, it follows from Lemmas \ref{lem003} and \ref{LEM860} that if $\omega>R^{\varepsilon_{0}},$
\begin{align}\label{QNE089}
&\sup\limits_{t\in(\bar{t},\bar{t}+\delta_{0} R^{p+\vartheta})}\int_{B_{(1-\varrho)R}}v^{2}w_{1}dx\notag\\
&\leq\int_{B_{R}}v^{2}(x,\bar{t})\zeta^{p}(x)w_{1}dx+C\int_{B_{R}\times[\bar{t},\bar{t}+\delta_{0}\omega^{2-p}R^{p+\vartheta}]}v^{p}(|\nabla\zeta|^{p}+|\zeta|^{p})w_{2}dxdt\notag\\
&\quad+C\big|A^{\delta_{0}}(\mu^{+}-2^{-1}\omega,R)\big|_{\nu_{w_{2}}}^{1-\frac{1}{l_{0}}}\notag\\
&\leq\frac{\omega^{2}}{8}|B_{R}|_{\mu_{w_{1}}}+\frac{C\omega^{p}}{(\varrho R)^{p}}\big|A^{\delta}(\mu^{+}-2^{-1}\omega,R)\big|_{\nu_{w_{2}}}+C\big|A^{\delta_{0}}(\mu^{+}-2^{-1}\omega,R)\big|_{\nu_{w_{2}}}^{1-\frac{1}{l_{0}}}\notag\\
&\leq\frac{\omega^{2}}{8}|B_{R}|_{\mu_{w_{1}}}+C\Bigg(\frac{\omega^{\frac{p(l_{0}-1)+2}{l_{0}}}}{\varrho^{p}R^{\frac{p(l_{0}-1)-n-\theta_{1}-\theta_{2}}{l_{0}}}}+1\Bigg)\big|A^{\delta_{0}}(\mu^{+}-2^{-1}\omega,R)\big|_{\nu_{w_{2}}}^{1-\frac{1}{l_{0}}}\notag\\
&\leq\frac{\omega^{2}}{4}|B_{R}|_{\mu_{w_{1}}}\left(\frac{1}{2}+\frac{C}{\varrho^{p}}\bigg(\frac{|A^{\delta_{0}}(\mu^{+}-2^{-1}\omega,R)|_{\nu_{w_{2}}}}{|B_{R}\times[\bar{t},\bar{t}+\omega^{2-p}R^{p+\vartheta}]|_{\nu_{w_{2}}}}\bigg)^{1-\frac{1}{l_{0}}}\right).
\end{align}
For every $t\in[\bar{t},\bar{t}+\delta_{0}\omega^{2-p}R^{p+\vartheta}]$, we have
\begin{align*}
&\int_{B_{(1-\varrho)R}}v^{2}(x,t)w_{1}dx\notag\\
&\geq\frac{\omega^{2}}{4}\big(1-2^{-(c_{0}-1)}\big)^{2}|B_{(1-\varrho)R}\cap\{u(\cdot,t)>\mu^{+}-2^{-c_{0}}\omega\}|_{\mu_{w_{1}}},
\end{align*}
where $c_{0}$ is defined by \eqref{CON09}. Substituting this into \eqref{QNE089}, we obtain
\begin{align*}
&|B_{(1-\varrho)R}\cap\{u(\cdot,t)>\mu^{+}-2^{-c_{0}}\omega\}|_{\mu_{w_{1}}}\notag\\
&\leq\frac{|B_{R}|_{\mu_{w_{1}}}}{(1-2^{-(c_{0}-1)})^{2}}\left(\frac{1}{2}+\frac{C}{\varrho^{p}}\bigg(\frac{|A^{\delta_{0}}(\mu^{+}-2^{-1}\omega,R)|_{\nu_{w_{2}}}}{|B_{R}\times[\bar{t},\bar{t}+\omega^{2-p}R^{p+\vartheta}]|_{\nu_{w_{2}}}}\bigg)^{1-\frac{1}{l_{0}}}\right).
\end{align*}
Observe that $|B_{R}\setminus B_{(1-\varrho)R}|_{\mu_{w_{1}}}\leq C\varrho|B_{R}|_{\mu_{w_{1}}}$. Take
\begin{align*}
\varrho=\bigg(\frac{|A^{\delta_{0}}(\mu^{+}-2^{-1}\omega,R)|_{\nu_{w_{2}}}}{|B_{R}\times[\bar{t},\bar{t}+\omega^{2-p}R^{p+\vartheta}]|_{\nu_{w_{2}}}}\bigg)^{\frac{l_{0}-1}{l_{0}(p+1)}}.
\end{align*}
In view of the value of $c_{0}$ in \eqref{CON09}, we have
\begin{align*}
&\frac{|B_{R}\cap\{u(\cdot,t)>\mu^{+}-2^{-c_{0}}\omega\}|_{\mu_{w_{1}}}}{|B_{R}|_{\mu_{w_{1}}}}\notag\\
&\leq\frac{1}{2(1-2^{-(c_{0}-1)})^{2}}+\frac{C}{\varrho^{p}}\left[\varrho^{p+1}+\bigg(\frac{|A^{\delta_{0}}(\mu^{+}-2^{-1}\omega,R)|_{\nu_{w_{2}}}}{|B_{R}\times[\bar{t},\bar{t}+\omega^{2-p}R^{p+\vartheta}]|_{\nu_{w_{2}}}}\bigg)^{1-\frac{1}{l_{0}}}\right]\notag\\
&\leq\frac{5}{8}+C_{\ast}\bigg(\frac{|A^{\delta_{0}}(\mu^{+}-2^{-1}\omega,R)|_{\nu_{w_{2}}}}{|B_{R}\times[\bar{t},\bar{t}+\omega^{2-p}R^{p+\vartheta}]|_{\nu_{w_{2}}}}\bigg)^{\frac{l_{0}-1}{l_{0}(p+1)}}\leq\frac{5}{8}+C_{\ast}\delta_{0}^{\frac{l_{0}-1}{l_{0}(p+1)}}.
\end{align*}
By picking $\delta_{0}=(8C_{\ast})^{-\frac{l_{0}-1}{l_{0}(p+1)}}$, we obtain
\begin{align*}
\frac{|B_{R}\cap\{u(\cdot,t)>\mu^{+}-2^{-c_{0}}\omega\}|_{\mu_{w_{1}}}}{|B_{R}|_{\mu_{w_{1}}}}\leq\frac{3}{4}.
\end{align*}
The proof is complete.

\end{proof}

From the assumed condition in \eqref{Q01}, we obtain that for any $\sigma\in(0,1]$,
\begin{align}\label{Q02}
|B_{R}\cap\{u(\cdot,\bar{t})>\mu^{+}-2^{-1}\sigma\omega\}|_{\mu_{w_{1}}}\leq2^{-1}|B_{R}|_{\mu_{w_{1}}}.
\end{align}
Since the constant $\delta_{0}$ captured in Lemma \ref{LEM0035} is independent of $\omega$, then by using \eqref{Q02} and replacing $\omega$ with $\sigma\omega$ in the above proof, we have either $\omega\leq \sigma^{-1}R^{\varepsilon_{0}}$, or
\begin{align}\label{Q03}
|B_{R}\cap\{u(\cdot,\bar{t}+\delta_{0}(\sigma\omega)^{2-p}R^{p+\vartheta})>\mu^{+}-2^{-c_{0}}\sigma\omega\}|_{\mu_{w_{1}}}\leq\frac{3}{4}|B_{R}|_{\mu_{w_{1}}},
\end{align}
where $c_{0}$ is defined by \eqref{CON09} and $\bar{t}\in[-1,-\delta_{0}(\sigma\omega)^{2-p}R^{p+\vartheta}]$. Introduce the following exponential variable substitution: for $\tau\in[0,\infty)$,
\begin{align}\label{EXP01}
\sigma=e^{-\frac{\tau}{p-2}},\quad\tilde{u}(x,\tau)=\frac{e^{\frac{\tau}{p-2}}}{\omega}(\delta_{0}R^{p+\vartheta})^{\frac{1}{p-2}}u(x,\bar{t}+e^{\tau}\omega^{2-p}\delta_{0}R^{p+\vartheta}).
\end{align}
Define $\tilde{\mu}^{+}:=\frac{e^{\frac{\tau}{p-2}}}{\omega}(\delta_{0}R^{p+\vartheta})^{\frac{1}{p-2}}\mu^{+}$. Then \eqref{Q03} implies that for any $\tau\geq0$,
\begin{align}\label{Q05}
|B_{R}\cap\{\tilde{u}(\cdot,\tau)>\tilde{\mu}^{+}-\kappa_{0}\}|_{\mu_{w_{1}}}\leq\frac{3}{4}|B_{R}|_{\mu_{w_{1}}},\quad \kappa_{0}:=\frac{(\delta_{0}R^{p+\vartheta})^{\frac{1}{p-2}}}{2^{c_{0}}}.
\end{align}
Assume without loss of generality that $u\leq0$. Otherwise, let $u-\sup\limits_{Q(1,1)}u$ substitute for $u$ in the following proofs. For simplicity, denote
\begin{align*}
\mathcal{G}:=\mathcal{G}(\tau)=\frac{e^{\frac{\tau}{p-2}}}{\omega}(\delta_{0}R^{p+\vartheta})^{\frac{1}{p-2}}.
\end{align*}
A direct calculation shows that
\begin{align*}
w_{1}\partial_{\tau}\tilde{u}=&\mathcal{G}^{p-1}w_{1}\partial_{t}u+\frac{\mathcal{G}w_{1}}{p-2}u\notag\\
\leq&\mathcal{G}^{p-1}[\mathrm{div}(w_{2}\mathbf{a}(x,t,u,\nabla u))+w_{2}\tilde{b}(x,t,u,\nabla u)]\notag\\
=&:\mathrm{div}(w_{2}\tilde{\mathbf{a}}(x,\tau,\tilde{u},\nabla \tilde{u}))+w_{2}\tilde{b}(x,\tau,\tilde{u},\nabla \tilde{u}),
\end{align*}
where $\tilde{\mathbf{a}}:B_{1}\times\mathbb{R}^{+}\times\mathbb{R}^{-}\times\mathbb{R}^{n}\rightarrow\mathbb{R}^{n}$ and $\tilde{b}:B_{1}\times\mathbb{R}^{+}\times\mathbb{R}^{-}\times\mathbb{R}^{n}\rightarrow\mathbb{R}$ are subject to the following structure conditions:
\begin{align*}
\begin{cases}
\tilde{\mathbf{a}}(x,\tau,\tilde{u},\nabla\tilde{u})\cdot\nabla\tilde{u}\geq\lambda_{1}|\nabla \tilde{u}|^{p}-\tilde{\phi}_{1}(x,\tau),\notag\\
|\tilde{\mathbf{a}}(x,\tau,\tilde{u},\nabla \tilde{u})|\leq\lambda_{2}|\nabla \tilde{u}|^{p-1}+\tilde{\phi}_{2}(x,\tau),\notag\\
|\tilde{b}(x,\tau,\tilde{u},\nabla \tilde{u})|\leq\lambda_{3}|\nabla \tilde{u}|^{p-1}+\tilde{\phi}_{3}(x,\tau).
\end{cases}
\end{align*}
Here $\tilde{\phi}_{i}$, $i=1,2,3$ are given by
\begin{align*}
\tilde{\phi}_{1}(x,\tau)=\mathcal{G}^{p}\bar{\phi}_{1}(x,\tau),\quad \tilde{\phi}_{i}(x,\tau)=\mathcal{G}^{p-1}\bar{\phi}_{i}(x,\tau),\quad i=2,3,
\end{align*}
where
\begin{align*}
\bar{\phi}_{i}(x,\tau)=\phi_{i}(x,\bar{t}+e^{\tau}\omega^{2-p}\delta_{0}R^{p+\vartheta}),\quad i=1,2,3.
\end{align*}
Similarly as before, write $\tilde{\phi}:=\tilde{\phi}_{1}+\tilde{\phi}_{2}^{\frac{p}{p-1}}+\tilde{\phi}_{3}^{\frac{p}{p-1}}$. Remark that the admissible time interval corresponding to the transformed solution $\tilde{u}$ becomes the infinite interval $[0,\infty)$, which allows us to establish the decay estimates and achieve the pointwise oscillation improvement of the solution over a large cylinder in the following.

Define $\widetilde{m}:=(\frac{\kappa_{0}}{2^{j_{\ast}}})^{2-p}$, where $j_{\ast}\geq1$ will be chosen later. For $R\in(0,\frac{1}{2}]$, we introduce the forward cylinders as follows:
\begin{align*}
Q^{+}(2R,\widetilde{m}(3R)^{p+\vartheta}):=B_{2R}\times(0,\widetilde{m}(3R)^{p+\vartheta}],
\end{align*}
and
\begin{align*}
\mathcal{Q}^{+}_{R}(\widetilde{m}):=B_{R}\times(\widetilde{m}R^{p+\vartheta},\widetilde{m}(3R)^{p+\vartheta}].
\end{align*}
We now establish the decaying estimates for the distribution function of $\tilde{u}$ as follows.
\begin{lemma}\label{lem005}
Assume as above. There exists a constant $\overline{C}_{0}>1$ depending only on the data and independent of $j_{\ast},\omega$ such that we have either $\omega\leq \mathcal{N}_{\ast}R^{\varepsilon_{0}}$ for some constant $\mathcal{N}_{\ast}:=\mathcal{N}_{\ast}(j_{\ast},\mathrm{data})>1$, or
\begin{align*}
\frac{|\mathcal{Q}^{+}_{R}(\widetilde{m})\cap\{\tilde{u}>\tilde{\mu}^{+}-\frac{\kappa_{0}}{2^{j_{\ast}}}\}|_{\nu_{w_{1}}}}{|\mathcal{Q}^{+}_{R}(\widetilde{m})|_{\nu_{w_{1}}}}\leq\frac{\overline{C}_{0}}{j_{\ast}^{\frac{p-p_{\ast}}{pp_{\ast}}}},
\end{align*}
where $p_{\ast}\in(1,p)$ is given by Lemma \ref{prop002}, $\kappa_{0}$ is defined by \eqref{Q05}.
\end{lemma}

\begin{proof}
For $i\geq0$, write $k_{i}=\tilde{\mu}^{+}-\frac{\kappa_{0}}{2^{i}}$ and
\begin{align*}
A_{i}(\tau)=B_{R}\cap\{\tilde{u}(\cdot,\tau)>k_{i}\},\quad A_{i}=\mathcal{Q}^{+}_{R}(\widetilde{m})\cap\{\tilde{u}>k_{i}\}.
\end{align*}
Making use of \eqref{Q05}, we have from Lemma \ref{LEM860} that
\begin{align}\label{E99}
|B_{R}\setminus A_{i}(\tau)|_{\mu_{w_{1}}}\geq\frac{1}{4}|B_{R}|_{\mu_{w_{1}}}\geq C(n,\theta_{1},\theta_{2})R^{n+\theta_{1}+\theta_{2}}.
\end{align}
Using \eqref{pro001}, we have
\begin{align*}
&(k_{i+1}-k_{i})^{p_{\ast}}|A_{i+1}(\tau)|_{\mu_{w_{1}}}^{p_{\ast}}|B_{R}\setminus A_{i}(\tau)|_{\mu_{w_{1}}}\notag\\
&\leq C(n,p,\theta_{1},\theta_{2})R^{p_{\ast}(n+\theta_{1}+\theta_{2}+1)}\int_{A_{i}(\tau)\setminus A_{i+1}(\tau)}|\nabla \tilde{u}|^{p_{\ast}}w_{1}dx,
\end{align*}
which, together with \eqref{E99}, reads that
\begin{align*}
|A_{i+1}(\tau)|_{\mu_{w_{1}}}\leq\frac{C2^{i}}{\kappa_{0}}R^{\frac{(n+\theta_{1}+\theta_{2})(p_{\ast}-1)}{p_{\ast}}+1}\bigg(\int_{A_{i}(\tau)\setminus A_{i+1}(\tau)}|\nabla \tilde{u}|^{p_{\ast}}w_{1}dx\bigg)^{\frac{1}{p_{\ast}}}.
\end{align*}
Integrating this from $\widetilde{m}R^{p+\vartheta}$ to $\widetilde{m}(3R)^{p+\vartheta}$ and utilizing H\"{o}lder's inequality, we deduce
\begin{align*}
&|A_{i+1}|_{\nu_{w_{1}}}\leq\frac{C2^{i}\widetilde{m}^{\frac{p_{\ast}-1}{p_{\ast}}}}{\kappa_{0}}R^{\frac{(n+p+\vartheta+\theta_{1}+\theta_{2})(p_{\ast}-1)}{p_{\ast}}+1}\bigg(\int_{A_{i}\setminus A_{i+1}}|\nabla \tilde{u}|^{p_{\ast}}w_{1}dxd\tau\bigg)^{\frac{1}{p_{\ast}}}.
\end{align*}
In view of $1<p_{\ast}<p$ and using H\"{o}lder's inequality again, we derive
\begin{align*}
&\bigg(\int_{A_{i}\setminus A_{i+1}}|\nabla \tilde{u}|^{p_{\ast}}w_{1}dxd\tau\bigg)^{\frac{1}{p_{\ast}}}\notag\\
&\leq\bigg(\int_{A_{i}\setminus A_{i+1}}|\nabla \tilde{u}|^{p}w_{2}dxd\tau\bigg)^{\frac{1}{p}}\bigg(\int_{A_{i}\setminus A_{i+1}}|x'|^{\frac{p\theta_{1}-p_{\ast}\theta_{3}}{p-p_{\ast}}}|x|^{\frac{p\theta_{2}-p_{\ast}\theta_{4}}{p-p_{\ast}}}dxd\tau\bigg)^{\frac{p-p_{\ast}}{pp_{\ast}}}\notag\\
&\leq R^{\frac{\vartheta}{p}}|A_{i}\setminus A_{i+1}|_{\nu_{w_{1}}}^{\frac{p-p_{\ast}}{pp_{\ast}}}\bigg(\int_{\mathcal{Q}^{+}_{R}(\widetilde{m})}|\nabla (\tilde{u}-k_{i})_{+}|^{p}w_{2}dxd\tau\bigg)^{\frac{1}{p}}.
\end{align*}
For simplicity, denote
\begin{align*}
\rho_{0}=2R,\quad\tau_{0}=\widetilde{m}(3R)^{p+\vartheta},\quad\text{and then }Q^{+}(\rho_{0},\tau_{0})=Q^{+}(2R,\widetilde{m}(3R)^{p+\vartheta}).
\end{align*}
Take a smooth cutoff function $\xi\in C^{\infty}(Q^{+}(\rho_{0},\tau_{0}))$ such that
\begin{align*}
\begin{cases}
0\leq\xi\leq1,&\mathrm{in}\;Q^{+}(\rho_{0},\tau_{0}),\\
\xi=1,&\mathrm{in}\;\mathcal{Q}^{+}_{R}(\widetilde{m}),\\
\xi=0,&\mathrm{on}\;\partial_{pa}Q^{+}(\rho_{0},\tau_{0}),\\
|\nabla\xi|\leq\frac{2}{R},\quad|\partial_{\tau}\xi|\leq\frac{2}{\widetilde{m}R^{p+\vartheta}},
\end{cases}
\end{align*}
where $\partial_{pa}Q^{+}(\rho_{0},\tau_{0})$ represents the parabolic boundary of $Q^{+}(\rho_{0},\tau_{0})$. Applying the proof of Lemma \ref{lem003} to $\tilde{u}$, we deduce from $\mathrm{(}\mathbf{K2}\mathrm{)}$ and Lemma \ref{LEM860} that if $\omega>\mathcal{N}_{\ast}(j_{\ast},\mathrm{data})R^{\varepsilon_{0}}$,
\begin{align*}
&\int_{\mathcal{Q}^{+}_{R}(\widetilde{m})}|\nabla (\tilde{u}-k_{i})_{+}|^{p}w_{2}dxd\tau\leq\int_{Q^{+}(\rho_{0},\tau_{0})}|\nabla((\tilde{u}-k_{i})_{+}\xi)|^{p}w_{2}dxd\tau\notag\\
&\leq C\int_{Q^{+}(\rho_{0},\tau_{0})}\big((\tilde{u}-k_{i})_{+}^{2}|\partial_{\tau}\xi||x'|^{\theta_{1}-\theta_{3}}|x|^{\theta_{2}-\theta_{4}}+(\tilde{u}-k_{i})_{+}^{p}(|\nabla\xi|^{p}+|\xi|^{p})\big)w_{2}\notag\\
&\quad+C\|\tilde{\phi}\|_{L^{l_{0}}(Q^{+}(\rho_{0},\tau_{0}),w_{2})}|Q^{+}(\rho_{0},\tau_{0})\cap \{\tilde{u}>k_{i}\}|_{\nu_{w_{2}}}^{1-\frac{1}{l_{0}}}\notag\\
&\leq\frac{C\kappa_{0}^{p}}{2^{pi}R^{p}}|Q^{+}(\rho_{0},\tau_{0})\cap \{\tilde{u}>k_{i}\}|_{\nu_{w_{2}}}\notag\\
&\quad+C\mathcal{G}(\widetilde{m}(3R)^{p+\vartheta})\|\bar{\phi}\|_{L^{l_{0}}(Q^{+}(\rho_{0},\tau_{0}),w_{2})}|Q^{+}(\rho_{0},\tau_{0})\cap \{\tilde{u}>k_{i}\}|_{\nu_{w_{2}}}^{1-\frac{1}{l_{0}}}\notag\\
&\leq \Bigg(\frac{C2^{\frac{j_{\ast}(p-2)}{l_{0}}}}{2^{pi}R^{\frac{p(l_{0}-1)-n-\theta_{1}-\theta_{2}}{l_{0}}}}+\Big(\frac{\mathcal{N}_{\ast}}{\omega}\Big)^{\frac{p(l_{0}-1)}{l_{0}}}\Bigg)\frac{|Q^{+}(\rho_{0},\tau_{0})\cap \{\tilde{u}>k_{i}\}|_{\nu_{w_{2}}}^{1-\frac{1}{l_{0}}}}{R^{-\frac{(p+\vartheta)(p(l_{0}-1)+2)}{l_{0}(p-2)}}}\notag\\
&\leq\frac{C\widetilde{m}\kappa_{0}^{p}}{2^{pi}}R^{n+\theta_{1}+\theta_{2}},
\end{align*}
where $C:=C(\text{data})$, $\mathcal{N}_{\ast}:=\mathcal{N}_{\ast}(j_{\ast},\text{data})$, and we also used the fact that
\begin{align}\label{DQ01}
\mathcal{G}(\widetilde{m}(3R)^{p+\vartheta})\leq \frac{\mathcal{C}_{\ast}R^{\frac{p(p+\vartheta)}{p-2}}}{\omega^{p}},\quad \mathcal{C}_{\ast}:=\mathcal{C}_{\ast}(j_{\ast},\mathrm{data})>1,
\end{align}
and
\begin{align}\label{DQ02}
\|\bar{\phi}\|_{L^{l_{0}}(Q^{+}(\rho_{0},\tau_{0}),w_{2})}\leq \Big(\frac{\omega^{p-2}}{\delta_{0}R^{p+\vartheta}}\Big)^{\frac{1}{l_{0}}}\|\phi\|_{L^{l_{0}}(Q(1,1),w_{2})}.
\end{align}
A combination of these above facts shows that
\begin{align*}
|A_{i+1}|_{\nu_{w_{1}}}\leq&C|A_{i}\setminus A_{i+1}|_{\nu_{w_{1}}}^{\frac{p-p_{\ast}}{pp_{\ast}}}|\mathcal{Q}^{+}_{R}(\widetilde{m})|_{\nu_{w_{1}}}^{\frac{pp_{\ast}+p_{\ast}-p}{pp_{\ast}}}.
\end{align*}
Hence we obtain that for $j\geq1$,
\begin{align*}
j|A_{j}|_{\nu_{w_{1}}}^{\frac{pp_{\ast}}{p-p_{\ast}}}\leq&\sum^{j-1}_{i=0}|A_{i+1}|_{\nu_{w_{1}}}^{\frac{pp_{\ast}}{p-p_{\ast}}}\leq C|\mathcal{Q}^{+}_{R}(\widetilde{m})|_{\nu_{w_{1}}}^{\frac{pp_{\ast}}{p-p_{\ast}}}.
\end{align*}
The proof is complete.

\end{proof}

Utilizing Lemma \ref{lem005}, we obtain the following pointwise oscillation improvement for the solution $\tilde{u}$.
\begin{lemma}\label{LEM090}
Assume as before. The constant $j_{\ast}$ can be chosen depending only on the data and independent of $\omega$ such that we derive either  $\omega\leq A_{0}R^{\varepsilon_{0}}$ for some large constant $A_{0}:=A_{0}(\mathrm{data})>1$, or
\begin{align*}
\tilde{u}(x,\tau)\leq \tilde{\mu}^{+}-\frac{\kappa_{0}}{2^{j_{\ast}+1}},\quad\mathrm{for }\;(x,\tau)\in B_{R/2}\times(\widetilde{m}(2R)^{p+\vartheta},\widetilde{m}(3R)^{p+\vartheta}],
\end{align*}
where $\kappa_{0}$ is defined by \eqref{Q05}.
\end{lemma}

\begin{proof}
For $i=0,1,2,...,$ write
\begin{align*}
r_{i}=\frac{R}{2}+\frac{R}{2^{i+1}},\quad \tilde{r}_{i}=\big(2^{-i}+2^{p+\vartheta}(1-2^{-i})\big)R^{p+\vartheta},
\end{align*}
and
\begin{align*}
k_{i}=\tilde{\mu}^{+}-\frac{\kappa_{0}}{2^{j_{\ast}+1}}-\frac{\kappa_{0}}{2^{j_{\ast}+1+i}}.
\end{align*}
Denote $\widetilde{\mathcal{Q}}_{i}^{+}(\widetilde{m}):=B_{r_{i}}\times(\widetilde{m}\tilde{r}_{i},\widetilde{m}(3R)^{p+\vartheta}].$ Choose a cutoff function $\xi_{i}\in C^{\infty}(\widetilde{\mathcal{Q}}_{i}^{+}(m))$ such that
\begin{align*}
\begin{cases}
0\leq\xi_{i}\leq1,&\mathrm{in}\;\widetilde{\mathcal{Q}}_{i}^{+}(\widetilde{m}),\\
\xi_{i}=1,&\mathrm{in}\;\widetilde{\mathcal{Q}}_{i+1}^{+}(\widetilde{m}),\\
\xi_{i}=0,&\mathrm{on}\;\partial_{pa}\widetilde{\mathcal{Q}}_{i}^{+}(\widetilde{m}),\\
|\nabla\xi_{i}|\leq\frac{2^{i+3}}{R},\quad|\partial_{\tau}\xi_{i}|\leq\frac{2^{i+2}}{(2^{p+\vartheta}-1)\widetilde{m}R^{p+\vartheta}},
\end{cases}
\end{align*}
where $\partial_{pa}\widetilde{\mathcal{Q}}_{i}^{+}(\widetilde{m})$ denotes the parabolic boundary of $\widetilde{\mathcal{Q}}_{i}^{+}(\widetilde{m})$. In light of $\mathrm{(}\mathbf{K2}\mathrm{)}$, Lemma \ref{LEM860} and \eqref{DQ01}--\eqref{DQ02}, it follows from the proof of Lemma \ref{lem003} with a slight modification that
\begin{align}\label{WQ932}
&\sup\limits_{\tau\in(\widetilde{m}r_{i}^{p+\vartheta},\widetilde{m}(3R)^{p+\vartheta})}m\int_{B_{r_{i}}}(\tilde{u}-k_{i})_{+}^{p}\xi_{i}^{p}w_{1}dx+\frac{\lambda_{1}}{3}\int_{\widetilde{\mathcal{Q}}_{i}^{+}(\widetilde{m})}|\nabla((\tilde{u}-k_{i})_{+}\xi_{i})|^{p}w_{2}dxd\tau\notag\\
&\leq\sup\limits_{\tau\in(\widetilde{m}r_{i}^{p+\vartheta},\widetilde{m}(3R)^{p+\vartheta})}\int_{B_{r_{i}}}(\tilde{u}-k_{i})_{+}^{2}\xi_{i}^{p}w_{1}dx+\frac{\lambda_{1}}{3}\int_{\widetilde{\mathcal{Q}}_{i}^{+}(\widetilde{m})}|\nabla((\tilde{u}-k_{i})_{+}\xi_{i})|^{p}w_{2}dxd\tau\notag\\
&\leq C\int_{\widetilde{\mathcal{Q}}_{i}^{+}(\widetilde{m})}\big((\tilde{u}-k_{i})_{+}^{2}|\partial_{\tau}\xi_{i}||x'|^{\theta_{1}-\theta_{3}}|x|^{\theta_{2}-\theta_{4}}+(\tilde{u}-k_{i})_{+}^{p}|\nabla\xi_{i}|^{p}\big)w_{2}dxd\tau\notag\\
&\quad+C\|\tilde{\phi}\|_{L^{l_{0}}(\widetilde{\mathcal{Q}}_{i}^{+}(\widetilde{m}),w_{2})}|\widetilde{\mathcal{Q}}_{i}^{+}(\widetilde{m})\cap \{\tilde{u}>k_{i}\}|_{\nu_{w_{2}}}^{1-\frac{1}{l_{0}}}\notag\\
&\leq \frac{C2^{pi}}{R^{p}}\left(\frac{\kappa_{0}}{2^{j_{\ast}}}\right)^{p}|\widetilde{\mathcal{Q}}_{i}^{+}(\widetilde{m})\cap \{\tilde{u}>k_{i}\}|_{\nu_{w_{2}}}\notag\\
&\quad+C\mathcal{G}(\widetilde{m}(3R)^{p+\vartheta})\|\bar{\phi}\|_{L^{l_{0}}(\widetilde{\mathcal{Q}}_{i}^{+}(\widetilde{m}),w_{2})}|\widetilde{\mathcal{Q}}_{i}^{+}(\widetilde{m})\cap \{\tilde{u}>k_{i}\}|_{\nu_{w_{2}}}^{1-\frac{1}{l_{0}}}\notag\\
&\leq\frac{C2^{pi}}{R^{p}}\left(\frac{\kappa_{0}}{2^{j_{\ast}}}\right)^{p}|\widetilde{\mathcal{Q}}_{i}^{+}(\widetilde{m})\cap \{\tilde{u}>k_{i}\}|_{\nu_{w_{2}}}\notag\\
&\quad+\frac{\overline{\mathcal{C}}_{\ast}R^{\frac{(p+\vartheta)(p(l_{0}-1)+2)}{l_{0}(p-2)}}}{\omega^{\frac{p(l_{0}-1)+2}{l_{0}}}}|\widetilde{\mathcal{Q}}_{i}^{+}(\widetilde{m})\cap \{\tilde{u}>k_{i}\}|_{\nu_{w_{2}}}^{1-\frac{1}{l_{0}}},
\end{align}
where $C=C(\mathrm{data})$ and $\overline{\mathcal{C}}_{\ast}=\overline{\mathcal{C}}_{\ast}(j_{\ast},\mathrm{data}).$ Define $\hat{u}(x,\hat{\tau})=\tilde{u}(x,\widetilde{m}\hat{\tau})$, $\hat{\xi}_{i}(x,\hat{\tau})=\xi_{i}(x,\widetilde{m}\hat{\tau})$, and
\begin{align*}
\hat{A}_{i}(\hat{\tau})=B_{r_{i}}\cap\{\hat{u}(\cdot,\hat{\tau})>k_{i}\},\quad\hat{A}_{i}=\widetilde{\mathcal{Q}}_{i}^{+}(1)\cap\{\hat{u}>k_{i}\}.
\end{align*}
Combining Proposition \ref{prop001} and \eqref{WQ932}, we obtain
\begin{align*}
&2^{-p(i+2)}\left(\frac{\kappa_{0}}{2^{j_{\ast}}}\right)^{p}|\hat{A}_{i+1}|_{\nu_{w_{2}}}^{\frac{1}{\chi}}=(k_{i+1}-k_{i})^{p}|\hat{A}_{i+1}|_{\nu_{w_{2}}}^{\frac{1}{\chi}}\notag\\
&\leq\|(\hat{u}_{i}-k_{i})_{+}\hat{\xi}_{i}\|^{p}_{L^{p\chi}(\widetilde{\mathcal{Q}}_{i}^{+}(1),w_{2})}\leq C\|(\hat{u}_{i}-k_{i})_{+}\hat{\xi}_{i}\|^{p}_{V^{p}_{0}(\widetilde{\mathcal{Q}}_{i}^{+}(1),w_{1},w_{2})}\notag\\
&\leq\frac{C2^{pi}}{R^{p}}\left(\frac{\kappa_{0}}{2^{j_{\ast}}}\right)^{p}|\hat{A}_{i}|_{\nu_{w_{2}}}+\frac{\widetilde{\mathcal{C}}_{\ast}R^{\frac{p(p+\vartheta)}{p-2}}}{\omega^{\frac{p(l_{0}-1)+2}{l_{0}}}}|\hat{A}_{i}|_{\nu_{w_{2}}}^{1-\frac{1}{l_{0}}},
\end{align*}
where $\chi=\frac{n+p+\theta_{1}+\theta_{2}}{n+\theta_{1}+\theta_{2}}$ and $\widetilde{\mathcal{C}}_{\ast}=\widetilde{\mathcal{C}}_{\ast}(j_{\ast},\mathrm{data})$. Consequently, if $\omega>\overline{\mathcal{N}}_{\ast}(j_{\ast},\mathrm{data})R^{\varepsilon_{0}}$, we have
\begin{align*}
|\hat{A}_{i+1}|_{\nu_{w_{2}}}\leq&\bigg(\frac{C4^{pi}}{R^{p}}|\hat{A}_{i}|_{\nu_{w_{2}}}+\Big(\frac{\overline{\mathcal{N}}(j_{\ast},\mathrm{data})}{\omega}\Big)^{\frac{p(l_{0}-1)+2}{l_{0}}}|\hat{A}_{i}|_{\nu_{w_{2}}}^{1-\frac{1}{l_{0}}}\bigg)^{\chi}\notag\\
\leq&\bigg[\bigg(\frac{C4^{pi}}{R^{\frac{p(l_{0}-1)-n-\theta_{1}-\theta_{2}}{l_{0}}}}+\Big(\frac{\overline{\mathcal{N}}(j_{\ast},\mathrm{data})}{\omega}\Big)^{\frac{p(l_{0}-1)+2}{l_{0}}}\bigg)|\hat{A}_{i}|_{\nu_{w_{2}}}^{1-\frac{1}{l_{0}}}\bigg]^{\chi}\notag\\
\leq&\bigg(\frac{C4^{pi}}{R^{\frac{p(l_{0}-1)-n-\theta_{1}-\theta_{2}}{l_{0}}}}|\hat{A}_{i}|_{\nu_{w_{2}}}^{1-\frac{1}{l_{0}}}\bigg)^{\chi}.
\end{align*}
Denote
\begin{align*}
F_{i}:=\frac{|\hat{A}_{i}|_{\nu_{w_{2}}}}{|B_{R}\times(R^{p+\vartheta},(3R)^{p+\vartheta}]|_{\nu_{w_{2}}}}.
\end{align*}
Therefore, we obtain
\begin{align*}
F_{i+1}\leq&(C4^{pi})^{\chi} F_{i}^{\frac{\chi(l_{0}-1)}{l_{0}}}\leq\prod\limits^{i}_{s=0}\big[(C4^{p(i-s)})^{\chi}\big]^{\big(\frac{\chi(l_{0}-1)}{l_{0}}\big)^{s}}F_{0}^{\left(\frac{\chi(l_{0}-1)}{l_{0}}\right)^{i+1}}\notag\\
\leq&(\widetilde{C}_{0}F_{0})^{\left(\frac{\chi(l_{0}-1)}{l_{0}}\right)^{i+1}},
\end{align*}
where $\widetilde{C}_{0}=\widetilde{C}_{0}(\mathrm{data})$. Fix the value of $j_{\ast}$ such that
\begin{align}\label{E98}
\frac{\overline{C}_{0}}{j_{\ast}^{\frac{p-p_{\ast}}{pp_{\ast}}}}\leq(2C_{0}\widetilde{C}_{0})^{-\frac{\beta}{n-1+\theta_{3}+\min\{0,\theta_{4}\}}},
\end{align}
where $\beta$ and $C_{0}$ are determined by applying Lemma \ref{lem09} to the domain $B_{R}\times(R^{p+\vartheta},(3R)^{p+\vartheta}]$, $\overline{C}_{0}$ is given by Lemma \ref{lem005}. Hence, it follows that
\begin{align*}
F_{i+1}\leq2^{-\left(\frac{\chi(l_{0}-1)}{l_{0}}\right)^{i+1}}\rightarrow0,\quad\mathrm{as}\;i\rightarrow\infty.
\end{align*}
The proof is then finished by letting $A_{0}:=A_{0}(\mathrm{data})=\overline{\mathcal{N}}_{\ast}(j_{\ast},\mathrm{data})$.

\end{proof}

Based on Lemma \ref{LEM090}, we are now ready to complete the proof of Proposition \ref{PRO01} by rescaling back to $u$.
\begin{proof}[Proof of Proposition \ref{PRO01}]
Take the proof of \eqref{A01} for example. Let
$$t=\bar{t}+e^{\tau}\omega^{2-p}\delta_{0}R^{p+\vartheta}.$$
When $\tau\in(\widetilde{m}(2R)^{p+\vartheta},\widetilde{m}(3R)^{p+\vartheta}]$, we have
\begin{align*}
\bar{t}+\delta_{0}e^{\delta_{0}^{-1}2^{p+\vartheta+(p-2)(j_{\ast}+c_{0})}}\omega^{2-p}R^{p+\vartheta}<t\leq\bar{t}+\delta_{0}e^{\delta_{0}^{-1}3^{p+\vartheta}2^{(p-2)(j_{\ast}+c_{0})}}\omega^{2-p}R^{p+\vartheta},
\end{align*}
where $c_{0},\delta_{0},j_{\ast}$ are, respectively, given by Lemma \ref{LEM0035} and \eqref{E98}. Pick
\begin{align}\label{ME09}
\bar{t}=-c_{\ast}\omega^{2-p}R^{p+\vartheta},\quad c_{\ast}=\delta_{0}e^{\delta_{0}^{-1}3^{p+\vartheta}2^{(p-2)(j_{\ast}+c_{0})}}.
\end{align}
Remark that the value of $\bar{t}$ chosen in \eqref{ME09} satisfies the requirement in \eqref{Q03}, that is,
\begin{align*}
\bar{t}\in[-1,-\delta_{0}(\sigma\omega)^{2-p}R^{p+\vartheta}],\quad\text{as }  \tau\in(\widetilde{m}(2R)^{p+\vartheta},\widetilde{m}(3R)^{p+\vartheta}]\text{ and }\omega>\sigma^{-1}R^{\varepsilon_{0}},
\end{align*}
where $\sigma=e^{-\frac{\tau}{p-2}}$. Choose
\begin{align}\label{M900}
M=(b_{\ast}2^{p+\vartheta})^{\frac{1}{p-2}},\quad\text{and thus }m=\Big(\frac{\omega}{M}\Big)^{2-p}=2^{p+\vartheta}b_{\ast}\omega^{2-p},
\end{align}
where
\begin{align*}
b_{\ast}=\delta_{0}e^{\delta_{0}^{-1}3^{p+\vartheta}2^{(p-2)(j_{\ast}+c_{0})}}-\delta_{0}e^{\delta_{0}^{-1}2^{p+\vartheta+(p-2)(j_{\ast}+c_{0})}}.
\end{align*}
Then we have $t\in(-m(R/2)^{p+\vartheta},0]$, as $\tau\in(\widetilde{m}(2R)^{p+\vartheta},\widetilde{m}(3R)^{p+\vartheta}]$. From \eqref{EXP01} and Lemma \ref{LEM090}, we have either $\omega\leq AR^{\varepsilon_{0}}:=\max\{A_{0},c_{\ast}^{\frac{1}{p-2}}\}R^{\varepsilon_{0}}$, or
\begin{align*}
u(x,t)\leq\mu^{+}-\frac{\omega}{2^{c_{0}+j_{\ast}+1}e^{\frac{\widetilde{m}(3R)^{p+\vartheta}}{p-2}}}=:\mu^{+}-\frac{\omega}{2^{\kappa_{\ast}}},
\end{align*}
for any $(x,t)\in Q(R/2,m(R/2)^{p+\vartheta})$, where
\begin{align*}
\kappa_{\ast}=c_{0}+j_{\ast}+1+\frac{3^{p+\vartheta}2^{(c_{0}+j_{\ast})(p-2)}}{\delta_{0}\ln2}.
\end{align*}
Therefore, \eqref{A01} is proved. By the same argument, we obtain that \eqref{A02} also holds.

\end{proof}

\subsection{The proofs of Theorems \ref{ZWTHM90} and \ref{THM060}.}
First, a direct application of Proposition \ref{PRO01} gives the following result.
\begin{lemma}\label{LEMD05}
Assume as in Theorems \ref{ZWTHM90} and \ref{THM060}. Then we have either $\omega\leq AR^{\varepsilon_{0}}$, or
\begin{align*}
\mathop{osc}\limits_{Q(R/2,m(R/2)^{p+\vartheta})}u\leq\eta^{\ast}\omega,\quad\eta^{\ast}=1-2^{-\kappa_{\ast}},
\end{align*}
where $A,m,\kappa_{\ast}$ are determined by Proposition \ref{PRO01}.

\end{lemma}
\begin{proof}
Note that one of the following two inequalities must hold:
\begin{align}\label{WAMQ001}
|B_{R}\cap\{u(\cdot,-c_{\ast}\omega^{2-p}R^{p+\vartheta})>\mu^{+}-2^{-1}\omega\}|_{\mu_{w_{1}}}\leq2^{-1}|B_{R}|_{\mu_{w_{1}}},
\end{align}
and
\begin{align}\label{WAMQ002}
|B_{R}\cap\{u(\cdot,-c_{\ast}\omega^{2-p}R^{p+\vartheta})<\mu^{-}+2^{-1}\omega\}|_{\mu_{w_{1}}}\leq2^{-1}|B_{R}|_{\mu_{w_{1}}}.
\end{align}
From Proposition \ref{PRO01}, it follows that if $\omega>AR^{\varepsilon_{0}}$,
\begin{align*}
\sup\limits_{Q(R/2,m(R/2)^{p+\vartheta})}u\leq\mu^{+}-2^{-\kappa_{\ast}}\omega,\quad\text{if \eqref{WAMQ001} holds,}
\end{align*}
and
\begin{align*}
\inf\limits_{Q(R/2,m(R/2)^{p+\vartheta})}u\geq\mu^{-}+2^{-\kappa_{\ast}}\omega,\quad\text{if \eqref{WAMQ002} holds.}
\end{align*}
In either case, we all obtain
\begin{align*}
\mathop{osc}\limits_{Q(R/2,m(R/2)^{p+\vartheta})}u\leq(1-2^{-k_{\ast}})\omega.
\end{align*}
The proof is finished.

\end{proof}

We proceed to use Lemma \ref{LEMD05} to construct a series of nested and shrinking cylinders with the same vertex such that the essential oscillation of $u$ in these cylinders goes to zero as the radius of the cylinder tends to zero. Denote
\begin{align}\label{K99}
\omega_{0}:=\max\{\omega,AR^{\varepsilon_{0}}\},
\end{align}
and, for $k\geq0$,
\begin{align}\label{K98}
R_{k}:=A^{-k}R,\quad \omega_{k+1}:=\max\{\eta^{\ast}\omega_{k},AR_{k}^{\varepsilon_{0}}\},\quad \tilde{a}_{k}:=\left(\frac{\omega_{k}}{A}\right)^{2-p}.
\end{align}
Since $A$ is a large constant, we have
\begin{align*}
\tilde{a}_{k+1}R_{k+1}^{p+\vartheta}=&\left(\frac{\omega_{k+1}}{A}\right)^{2-p}\frac{R_{k}^{p+\vartheta}}{A^{p+\vartheta}}\leq\frac{\tilde{a}_{k}R_{k}^{p+\vartheta}}{(\eta^{\ast})^{p-2}A^{p+\vartheta}}< \tilde{a}_{k}R_{k}^{p+\vartheta},
\end{align*}
which implies that $Q(R_{k+1},\tilde{a}_{k+1}R_{k+1}^{p+\vartheta})\subset Q(R_{k},\tilde{a}_{k}R_{k}^{p+\vartheta})$. Remark that $\tilde{a}_{0}\leq a_{0}$.
\begin{lemma}\label{LEM83}
Assume as in Theorems \ref{ZWTHM90} and \ref{THM060}. Then for any $k=0,1,2,...,$
\begin{align}\label{OSC98}
\mathop{osc}\limits_{Q(R_{k},\tilde{a}_{k}R_{k}^{p+\vartheta})}u\leq\omega_{k}.
\end{align}

\end{lemma}
\begin{proof}
Observe first that \eqref{OSC98} holds obviously for $k=0$ in virtue of $\tilde{a}_{0}\leq a_{0}$. We now suppose that \eqref{OSC98} holds in the case of $k=i$ for any given $i\geq1$. Then we prove that it also holds for $k=i+1$. Since $\mathop{osc}\limits_{Q(R_{i},\tilde{a}_{i}R_{i}^{p+\vartheta})}u\leq\omega_{i}$, it then follows from the proof of Lemma \ref{LEMD05} with minor modification that
\begin{align}\label{AEM81}
\mathop{osc}\limits_{Q(R_{i}/2,m_{i}(R_{i}/2)^{p+\vartheta})}u\leq\max\{\eta^{\ast}\omega_{i},AR_{i}^{\varepsilon_{0}}\}=\omega_{i+1},\quad m_{i}:=\left(\frac{M}{\omega_{i}}\right)^{p-2},
\end{align}
where $M$ is given by \eqref{M900}. Due to the fact that $\omega_{i+1}\geq\eta^{\ast}\omega_{i}$ and $A$ is a large constant, we obtain
\begin{align*}
&m_{i}\left(\frac{R_{i}}{2}\right)^{p+\vartheta}=\left(\frac{M}{\omega_{i}}\right)^{p-2}\left(\frac{R_{i}}{2}\right)^{p+\vartheta}\geq \left(\frac{A}{\omega_{i+1}}\right)^{p-2}\left(\frac{\eta^{\ast}M}{A}\right)^{p-2}\left(\frac{R_{i}}{2}\right)^{p+\vartheta}\notag\\
&=\tilde{a}_{i+1}R_{i+1}^{p+\vartheta}\frac{(\eta^{\ast}M)^{p-2}A^{2+\vartheta}}{2^{p+\vartheta}}\geq \tilde{a}_{i+1}R_{i+1}^{p+\vartheta},
\end{align*}
which, in combination with \eqref{AEM81}, shows that
\begin{align*}
\mathop{osc}\limits_{Q(R_{i+1},\tilde{a}_{i+1}R_{i+1}^{p+\vartheta})}u\leq\omega_{i+1}.
\end{align*}
The proof is finished.

\end{proof}

Based on the result in Lemma \ref{LEM83}, we now give a more precise characterization for the oscillation decay property of the solution $u$.
\begin{prop}\label{PRO90}
Suppose as in Theorems \ref{ZWTHM90} and \ref{THM060}. Then for any $0<\rho\leq R\leq\frac{1}{2}$,
\begin{align*}
\mathop{osc}\limits_{Q(\rho.\tilde{a}_{0}\rho^{p+\vartheta})}u\leq \max\big\{(\eta^{\ast})^{-1}\omega_{0},A^{1+2\varepsilon_{\ast}}R^{\varepsilon_{\ast}}\big\}\left(\frac{\rho}{R}\right)^{\varepsilon_{\ast}},
\end{align*}
where $\eta^{\ast}$ is given in Lemma \ref{LEMD05}, $\omega_{0}$ and $\tilde{a}_{0}$ are, respectively, defined by \eqref{K99}--\eqref{K98}, and
\begin{align}\label{Ast01}
\varepsilon_{\ast}:=\min\Big\{\varepsilon_{0},-\frac{\ln\eta^{\ast}}{\ln A}\Big\}.
\end{align}

\end{prop}

\begin{remark}\label{W98}
Observe that the value of $\varepsilon_{0}$ tends to zero, as $\theta_{1}+\theta_{2}\nearrow p(l_{0}-1)-n$ or $l_{0}\searrow\frac{n+p+\theta_{1}+\theta_{2}}{p}$. Then if $\varepsilon_{0}\leq-\frac{\ln\eta^{\ast}}{\ln A}$, $\varepsilon_{\ast}$ becomes the explicit exponent $\varepsilon_{0}$, whose value is clearly determined by the structure of the considered equation and the weights. By contrast, if $\varepsilon_{0}>-\frac{\ln\eta^{\ast}}{\ln A}$, the effects of the weights on the regularity of the solution will be concealed beneath the inexplicit constant $-\frac{\ln\eta^{\ast}}{\ln A}$.

\end{remark}

\begin{remark}\label{REM10}
Recall that for any $R\in(0,\frac{1}{2}]$ and $t_{0}\in[-\frac{1}{2},0]$, if $\omega>AR^{\varepsilon_{0}}$, we have
\begin{align*}
[(0,t_{0})+Q(R,a_{0}R^{p+\vartheta})]\subset[(0,t_{0})+Q(2R,R^{p+\vartheta-\bar{\varepsilon}_{0}})]\subset B_{1}\times(-1,0].
\end{align*}
Therefore, by a translation, it follows from the proof of Proposition \ref{PRO90} with a slight modification that for any $0<\rho\leq R\leq\frac{1}{2}$, there holds either $\omega\leq AR^{\varepsilon_{0}},$ or
\begin{align*}
\mathop{osc}\limits_{[(0,t_{0})+Q(\rho,\tilde{a}_{0}\rho^{p+\vartheta})]}u\leq \max\big\{(\eta^{\ast})^{-1}\omega_{0},A^{1+2\varepsilon_{\ast}}R^{\varepsilon_{\ast}}\big\}\left(\frac{\rho}{R}\right)^{\varepsilon_{\ast}}.
\end{align*}
\end{remark}

\begin{proof}[Proof of Proposition \ref{PRO90}]
For any $0<\rho\leq R\leq\frac{1}{2}$, there exists an integer $k\geq0$ such that $R_{k+1}=A^{-(k+1)}R\leq\rho\leq A^{-k}R=R_{k}$. In light of \eqref{Ast01}, we have $A^{\varepsilon_{\ast}}\eta^{\ast}\leq1$. This, together with \eqref{OSC98}, shows that the conclusion obviously holds for $k=0$. So in the following it suffices to consider the case when $k\geq1$. Utilizing the fact of $A^{\varepsilon_{\ast}}\eta^{\ast}\leq1$ again, we deduce
\begin{align}\label{D31}
\omega_{k}\leq&\max\{\eta^{\ast}\omega_{k-1},AR_{k-1}^{\varepsilon_{\ast}}\}\leq\max\big\{(\eta^{\ast})^{k}\omega_{0},\max\limits_{0\leq i\leq k-1}A(\eta^{\ast})^{i}R_{k-1-i}^{\varepsilon_{\ast}}\big\}\notag\\
\leq&\max\big\{(\eta^{\ast})^{k}\omega_{0},\max\limits_{0\leq i\leq k-1}R^{\varepsilon_{\ast}}(\eta^{\ast})^{i}A^{1-(k-1-i)\varepsilon_{\ast}}\big\}\notag\\
\leq&\max\bigg\{(\eta^{\ast})^{k}\omega_{0},A\left(\frac{R}{A^{k-1}}\right)^{\varepsilon_{\ast}}\bigg\}.
\end{align}
Since $\varepsilon_{\ast}\leq-\frac{\ln\eta^{\ast}}{\ln A}$ and $R_{k+1}\leq\rho$, then
\begin{align*}
(\eta^{\ast})^{k}\leq(\eta^{\ast})^{-1}A^{-(k+1)\varepsilon_{\ast}}\leq(\eta^{\ast})^{-1}\left(\frac{\rho}{R}\right)^{\varepsilon_{\ast}},\quad A\left(\frac{R}{A^{k-1}}\right)^{\varepsilon_{\ast}}\leq&A^{1+2\varepsilon_{\ast}}\rho^{\varepsilon_{\ast}}.
\end{align*}
Inserting this into \eqref{D31}, we have
\begin{align*}
\omega_{k}\leq\max\big\{(\eta^{\ast})^{-1}\omega_{0},A^{1+2\varepsilon_{\ast}}R^{\varepsilon_{\ast}}\big\}\left(\frac{\rho}{R}\right)^{\varepsilon_{\ast}}.
\end{align*}
Since
\begin{align*}
\omega_{k}=&\max\{\eta^{\ast}\omega_{k-1},AR_{k-1}^{\varepsilon_{0}}\}=\max\big\{(\eta^{\ast})^{k}\omega_{0},\max\limits_{0\leq i\leq k-1}A(\eta^{\ast})^{i}R_{k-1-i}^{\varepsilon_{0}}\big\}\notag\\
=&\max\big\{(\eta^{\ast})^{k}\omega_{0},\max\limits_{0\leq i\leq k-1}R^{\varepsilon_{0}}(\eta^{\ast})^{i}A^{1-(k-1-i)\varepsilon_{0}}\big\}\leq\omega_{0},
\end{align*}
then we have $\tilde{a}_{0}\leq \tilde{a}_{k}$. This implies that
\begin{align*}
\mathop{osc}\limits_{Q(\rho.\tilde{a}_{0}\rho^{p+\vartheta})}u\leq\mathop{osc}\limits_{Q(R_{k},\tilde{a}_{k}R_{k}^{p+\vartheta})}u\leq \max\big\{(\eta^{\ast})^{-1}\omega_{0},A^{1+2\varepsilon_{\ast}}R^{\varepsilon_{\ast}}\big\}\left(\frac{\rho}{R}\right)^{\varepsilon_{\ast}}.
\end{align*}
The proof is complete.

\end{proof}

In the following we make use of Proposition \ref{PRO90} to complete the proofs of Theorems \ref{ZWTHM90} and \ref{THM060}, respectively.
\begin{proof}[Proof of Theorem \ref{ZWTHM90}]
Denote $\hat{a}_{0}:=2^{\varepsilon_{0}(p-2)}$. Then we obtain that for any $0<R\leq\frac{1}{2}$,
\begin{align*}
\hat{a}_{0}=2^{\varepsilon_{0}(p-2)}=\left(\frac{A}{\max\{2\mathcal{M},A2^{-\varepsilon_{0}}\}}\right)^{p-2}\leq\left(\frac{A}{\max\{\omega,AR^{\varepsilon_{0}}\}}\right)^{p-2}=\tilde{a}_{0}.
\end{align*}
Then using Proposition \ref{PRO90} and Remark \ref{REM10}, we deduce that there exists a constant $0<\alpha\leq\varepsilon_{0}$ depending only on the data such that for any $t_{0}\in(-1/2,0]$ and $\rho\in(0,\frac{1}{2}]$,
\begin{align*}
\mathop{osc}\limits_{[(0,t_{0})+Q(\rho,\hat{a}_{0}\rho^{p+\vartheta})]}u\leq C\rho^{\alpha},
\end{align*}
which leads to that for every $(x,t)\in B_{1/2}\times(-1/2,t_{0}]$,

$(i)$ if $|t-t_{0}|\leq \hat{a}_{0}2^{-(p+\vartheta)}$,
\begin{align*}
|u(x,t)-u(0,t_{0})|\leq&|u(x,t)-u(x,t_{0})|+|u(x,t_{0})-u(0,t_{0})|\notag\\
\leq& C\big((\hat{a}_{0}^{-1}|t-t_{0}|)^{\frac{\alpha}{p+\theta}}+|x|^{\alpha}\big)\leq C\big(|x|+|t-t_{0}|^{\frac{1}{p+\vartheta}}\big)^{\alpha};
\end{align*}

$(ii)$ if $|t-t_{0}|>\hat{a}_{0}2^{-(p+\theta)}$, there exists an increasing set $\{t_{i}\}_{i=1}^{N}$, $1\leq N\leq\big[\hat{a}_{0}^{-1}2^{p+\vartheta-1}]+1$ such that $t<t_{1}\leq\cdots\leq t_{N}<t_{0}$,
\begin{align*}
|u(x,t)-u(0,t_{0})|\leq&|u(x,t)-u(x,t_{1})|+|u(x,t_{1})-u(x,t_{0})|+|u(x,t_{0})-u(0,t_{0})|\notag\\
\leq&C\big((\hat{a}_{0}^{-1}|t-t_{1}|)^{\frac{\alpha}{p+\vartheta}}+(\hat{a}_{0}^{-1}|t_{1}-t_{0}|)^{\frac{\alpha}{p+\vartheta}}+|x|^{\alpha}\big)\notag\\
\leq&C\big(|x|+|t-t_{0}|^{\frac{1}{p+\vartheta}}\big)^{\alpha},\quad\text{if}\;N=1,
\end{align*}
and
\begin{align*}
&|u(x,t)-u(0,t_{0})|\notag\\
&\leq|u(x,t)-u(x,t_{1})|+\sum^{N-1}_{i=1}|u(x,t_{i})-u(x,t_{i+1})|\notag\\
&\quad+|u(x,t_{N})-u(x,t_{0})|+|u(x,t_{0})-u(0,t_{0})|\notag\\
&\leq C\Big((\hat{a}_{0}^{-1}|t-t_{1}|)^{\frac{\alpha}{p+\vartheta}}+\sum^{N-1}_{i=1}(\hat{a}_{0}^{-1}|t_{i}-t_{i+1}|)^{\frac{\alpha}{p+\vartheta}}+(\hat{a}_{0}^{-1}|t_{N}-t_{0}|)^{\frac{\alpha}{p+\vartheta}}+|x|^{\alpha}\Big)\notag\\
&\leq C\big(|x|+|t-t_{0}|^{\frac{1}{p+\vartheta}}\big)^{\alpha},\quad \text{if}\;N\geq2.
\end{align*}
The proof is finished.

\end{proof}

\begin{proof}[Proof of Theorem \ref{THM060}]
In the following we take the proof in the case of $(w_{1},w_{2})=(|x'|^{\theta_{1}},|x'|^{\theta_{3}})$ for example. Another case is the same and thus omitted. In this case, we have $\theta_{2}=\theta_{4}=0$ and $\vartheta=\theta_{1}-\theta_{3}$. By a translation, we deduce from the proof of Theorem \ref{ZWTHM90} with a slight modification that there exists two constants $0<\alpha\leq\varepsilon_{0}$ and $C>1$, both depending only on the data, such that for any given $\bar{x}\in\{(0',\bar{x}_{n}):|\bar{x}_{n}|\leq1/2\}$ and $t_{0}\in (-1/2,0]$,
\begin{align}\label{QM916}
|u(x,t)-u(\bar{x},t_{0})|\leq C\big(|x-\bar{x}|+|t-t_{0}|^{\frac{1}{p+\vartheta}}\big)^{\alpha},
\end{align}
for all $(x,t)\in B_{1/4}(\bar{x})\times(-1/2,t_{0}]$. For $R\in(0,1/4)$ and $(y,s)\in Q(1/R,1/R^{p+\vartheta})$, let $u_{R}(y,s)=u(Ry,R^{p+\vartheta}s)$ and $\phi_{3,R}(y,s)=\phi_{3}(Ry,R^{p+\vartheta}s)$. Then $u_{R}$ solves
\begin{align*}
|y'|^{\theta_{1}}\partial_{s}u_{R}-\mathrm{div}(|y'|^{\theta_{3}}|\nabla u_{R}|^{p-2}\nabla u_{R})=|y'|^{\theta_{3}}\left(\lambda_{3}|\nabla u_{R}|^{p}+R^{p}\phi_{3,R}\right),
\end{align*}
for $(y,s)\in Q(1/R,1/R^{p+\vartheta})$.

For any given $(x,t),(\tilde{x},\tilde{t})\in B_{1/4}\times(-1/2,0],$ suppose without loss of generality that $|\tilde{x}'|\leq|x'|$. Let $R=|x'|$. Then applying the proof of Theorem \ref{ZWTHM90} with minor modification, we obtain that there exist two constants $0<\gamma<1$ and $C>1$, both depending only upon the data, such that for any fixed $\bar{y}\in B_{1/(2R)}\cap\{|y'|=1\}$ and $\bar{s}\in(-2^{-1} R^{-(p+\vartheta)},0]$,
\begin{align}\label{WAQA001}
|u_{R}(y,s)-u_{R}(\bar{y},\bar{s})|\leq C\Big(|y-\bar{y}|+\big(\hat{a}_{0}^{-1}|s-\bar{s}|\big)^{1/p}\Big)^{\gamma},
\end{align}
for any $(y,s)$ satisfying that $|y-\bar{y}|+(\hat{a}_{0}^{-1}|s-\bar{s}|)^{1/p}\leq1/4$. For later use, we limit the range of $\gamma$ to be in $(0,\alpha]$. Otherwise, if $\gamma>\alpha$, then \eqref{WAQA001} also holds for any $\gamma\in(0,\alpha]$.

Note that
\begin{align}\label{QT01}
|u(x,t)-u(\tilde{x},\tilde{t})|\leq&|u(x,t)-u(x,\tilde{t})|+|u(x,\tilde{t})-u(x',\tilde{x}_{n},\tilde{t})|\notag\\
&+|u(x',\tilde{x}_{n},\tilde{t})-u(\tilde{x},\tilde{t})|.
\end{align}
Set $c_{1}\geq1+\frac{p}{p+\vartheta}$. If $|t-\tilde{t}|\leq \hat{a}_{0}R^{c_{1}(p+\vartheta)}$, it follows from \eqref{WAQA001} that
\begin{align*}
&|u(x,t)-u(x,\tilde{t})|=\left|u_{R}(x/R,t/R^{p+\vartheta})-u_{R}(x/R,\tilde{t}/R^{p+\vartheta})\right|\notag\\
&\leq C|(t-\tilde{t})/(\hat{a}_{0}R^{p+\vartheta})|^{\frac{\gamma}{p}}\leq C|t-\tilde{t}|^{\frac{(c_{1}-1)\gamma}{c_{1}p}}.
\end{align*}
By contrast, if $|t-\tilde{t}|>\hat{a}_{0}R^{c_{1}(p+\vartheta)}$, using \eqref{QM916}, we obtain
\begin{align*}
&|u(x,t)-u(x,\tilde{t})|\notag\\
&\leq|u(x,t)-u(0',x_{n}.t)|+|u(0',x_{n},t)-u(0',x_{n},\tilde{t})|+|u(0',x_{n},\tilde{t})-u(x,\tilde{t})|\notag\\
&\leq C\big(R^{\alpha}+(\hat{a}_{0}^{-1}|t-\tilde{t}|)^{\frac{\alpha}{p+\vartheta}}\big)\leq C|t-\tilde{t}|^{\frac{\alpha}{c_{1}(p+\vartheta)}}.
\end{align*}

We now proceed to deal with the remaining two terms in the right hand side of \eqref{QT01}. Choose $c_{2}\geq2$. If $|x_{n}-\tilde{x}_{n}|\leq R^{c_{2}}$, we have from \eqref{WAQA001} that
\begin{align*}
&|u(x,\tilde{t})-u(x',\tilde{x}_{n},\tilde{t})|=\left|u_{R}(x/R,\tilde{t}/R^{p+\vartheta})-u_{R}(x'/R,\tilde{x}_{n}/R,\tilde{t}/R^{p+\vartheta})\right|\notag\\
&\leq C|(x_{n}-\tilde{x}_{n})/R|^{\gamma}\leq C|x_{n}-\tilde{x}_{n}|^{\frac{(c_{2}-1)\gamma}{c_{2}}},
\end{align*}
while, if $|x_{n}-\tilde{x}_{n}|>R^{c_{2}}$, it follows from \eqref{QM916} that
\begin{align*}
&|u(x,\tilde{t})-u(x',\tilde{x}_{n},\tilde{t})|\notag\\
&\leq|u(x,\tilde{t})-u(0',x_{n},\tilde{t})|+|u(0',x_{n},\tilde{t})-u(0',\tilde{x}_{n},\tilde{t})|+|u(0',\tilde{x}_{n},\tilde{t})-u(x',\tilde{x}_{n},\tilde{t})|\notag\\
&\leq C(|x'|^{\alpha}+|x_{n}-\tilde{x}_{n}|^{\alpha})\leq C|x_{n}-\tilde{x}_{n}|^{\frac{\alpha}{c_{2}}}.
\end{align*}
Similarly, if $|x'-\tilde{x}'|\leq R^{c_{2}}$, we obtain from \eqref{WAQA001} that
\begin{align*}
&|u(x',\tilde{x}_{n},\tilde{t})-u(\tilde{x},\tilde{t})|=\left|u_{R}(x'/R,\tilde{x}_{n}/R,\tilde{t}/R^{p+\vartheta})-u_{R}(\tilde{x}/R,\tilde{t}/R^{p+\vartheta})\right|\notag\\
&\leq C|(x'-\tilde{x}')/R|^{\gamma}\leq C|x'-\tilde{x}'|^{\frac{(c_{2}-1)\gamma}{c_{2}}},
\end{align*}
while, if $|x'-\tilde{x}'|>R^{c_{2}}$, utilizing \eqref{QM916}, it follows that
\begin{align*}
&|u(x',\tilde{x}_{n},\tilde{t})-u(\tilde{x},\tilde{t})|\leq|u(x',\tilde{x}_{n},\tilde{t})-u(0',\tilde{x}_{n},\tilde{t})|+|u(0',\tilde{x}_{n},\tilde{t})-u(\tilde{x},\tilde{t})|\notag\\
&\leq C\big(R^{\alpha}+|\tilde{x}'|^{\alpha}\big)\leq C|x'-\tilde{x}'|^{\frac{\alpha}{c_{2}}}.
\end{align*}
%Due to these above facts, we may decrease the value of $\gamma$ to make $\gamma\in(0,\frac{p\alpha}{p+\vartheta}]$ if necessary and
Based on these above facts, we pick $c_{1}=1+\frac{p\alpha}{\gamma(p+\vartheta)}$ and $c_{2}=1+\frac{\alpha}{\gamma}$ such that
\begin{align}\label{D87}
\frac{(c_{1}-1)\gamma}{c_{1}p}=\frac{\alpha}{c_{1}(p+\vartheta)},\quad \frac{(c_{2}-1)\gamma}{c_{2}}=\frac{\alpha}{c_{2}}.
\end{align}
Remark that since $\frac{c_{i}-1}{c_{i}}$ increases in $c_{i}$ and $c_{i}^{-1}$ decreases in $c_{i}$ for $i=1,2,$ then the values of $c_{1}$ and $c_{2}$ taken in \eqref{D87} maximize the H\"{o}lder regularity exponent and thus the choice is best. Therefore, we obtain that for any $(x,t),(\tilde{x},\tilde{t})\in B_{1/4}\times(-1/2,0),$
\begin{align*}
|u(x,t)-u(\tilde{x},\tilde{t})|\leq C\big(|x-\tilde{x}|+|t-\tilde{t}|^{\frac{1}{p+\vartheta}}\big)^{\frac{\alpha\gamma}{\alpha+\gamma}}.
\end{align*}
This leads to that Theorem \ref{THM060} holds.

\end{proof}

\section{Appendix:\,An alternative proof for the expansion of time}
\begin{prop}\label{LEM30}
Assume as in Lemma \ref{LEM0035}. Then there exists two constants $A_{1},j_{0}>2$ depending only on the data such that we have either $\omega\leq A_{1}R^{\varepsilon_{0}}$, or if for some $\bar{t}\in[-1,-\omega^{2-p}R^{p+\vartheta}]$,
\begin{align}\label{QE010}
|B_{R}\cap\{u(\cdot,\bar{t})>\mu^{+}-2^{-1}\omega\}|_{\mu_{w_{1}}}\leq2^{-1}|B_{R}|_{\mu_{w_{1}}},
\end{align}
then
\begin{align*}
|B_{R/2}\cap\{u(\cdot,t)>\mu^{+}-2^{-(j_{0}+1)}\omega\}|_{\mu_{w_{1}}}\leq\frac{3}{4}|B_{R/2}|_{\mu_{w_{1}}},\;\,\forall\,t\in[\bar{t},\bar{t}+\omega^{2-p}R^{p+\vartheta}].
\end{align*}

\end{prop}

We prepare to use the following logarithmic estimates to prove Proposition \ref{LEM30}. For that purpose, we first introduce the logarithmic function as follows:
\begin{align}\label{P998}
\Psi^{\pm}_{k,\delta}(u):=\Psi(H_{k}^{\pm},(u-k)_{\pm},\delta)=\ln^{+}\frac{H_{k}^{\pm}}{H_{k}^{\pm}-(u-k)_{\pm}+\delta},\quad0<\delta<H_{k}^{\pm},
\end{align}
where $H_{k}^{\pm}=\sup\limits_{[(x_{0},t_{0})+Q(\rho,\tau)]}(u-k)_{\pm}$, $\ln^{+}$ means that $\ln^{+}s=\max\{\ln s,0\}$ for $s>0$. For any fixed $B_{\rho}(x_{0})\subset B_{1}$, let $\zeta\in C^{\infty}(B_{1})$ be a smooth cutoff function satisfying that
\begin{align}\label{CUT01}
0\leq\zeta\leq1,\;\,|\nabla\zeta|<\infty\;\,\mathrm{in}\;B_{1},\quad \mathrm{and}\;\,\zeta=0\;\,\mathrm{in}\;B_{1}\setminus B_{\rho}(x_{0}).
\end{align}
The required logarithmic inequalities are now stated as follows.
\begin{lemma}\label{lem8021}
Let $u$ be the solution to problem \eqref{PO001} with $\Omega\times(-T,0]=B_{1}\times(-1,0]$. Then for any cylinder $[(x_{0},t_{0})+Q(\rho,\tau)]\subset B_{1}\times(-1,0]$, we obtain
\begin{align*}
&\sup\limits_{t_{0}-\tau<t<t_{0}}\int_{B_{\rho}(x_{0})}\Psi^{\pm}_{k,\delta}(u)(x,t)\zeta^{p}(x)w_{1}dx\notag\\
&\leq\int_{B_{\rho}(x_{0})}\Psi^{\pm}_{k,\delta}(u)(x,t_{0}-\tau)\zeta^{p}(x)w_{1}dx\notag\\
&\quad+C\int_{[(x_{0},t_{0})+Q(\rho,\tau)]}\Psi^{\pm}_{k,\delta}(u)|\partial_{u}\Psi^{\pm}_{k,\delta}(u)|^{2-p}(|\nabla\zeta|^{p}+|\zeta|^{p})w_{2}dxdt\notag\\
&\quad+\frac{C}{\delta^{2}}\left(1+\ln\frac{H_{k}^{\pm}}{\delta}\right)\|\phi\|_{L^{l_{0}}(B_{1}\times(-1,0),w_{2})}|[(x_{0},t_{0})+Q(\rho,\tau)]\cap \{v_{\pm}>0\}|_{\nu_{w_{2}}}^{1-\frac{1}{l_{0}}},
\end{align*}
where $\zeta$ is defined by \eqref{CUT01} and $\phi=\phi_{1}+\phi_{2}^{\frac{p}{p-1}}+\phi_{3}^{\frac{p}{p-1}}$.

\end{lemma}

\begin{proof}
Without loss of generality, assume that $(x_{0},t_{0})=(0,0).$ For simplicity, denote
\begin{align*}
\psi(f):=\Psi^{\pm}_{k,\delta}(f),\quad\psi':=\partial_{f}\psi,\quad f=u\;\mathrm{or}\;u_{h}.
\end{align*}
Set $\varphi=[\psi^{2}(u_{h})]'\zeta^{p}$. Let $\chi_{\Sigma}$ represent the characteristic function of the set $\Sigma$. It follows from a straightforward computation that
\begin{align*}
\psi'(u_{h})=\frac{\pm\chi_{\{(u_{h}-k)_{\pm}>0\}}}{H_{k}^{\pm}-(u_{h}-k)_{\pm}+\delta},\quad\psi''(u_{h})=[\psi'(u_{h})]^{2},
\end{align*}
and
\begin{align*}
[\psi^{2}(u_{h})]''=2(1+\psi(u_{h}))\psi'^{2}(u_{h})\in L_{loc}^{\infty}(B_{1}\times(-1,0)),
\end{align*}
which implies that $\varphi$ is an admissible test function. Similarly as before, we obtain that for any $-\tau\leq s\leq0$,
\begin{align}\label{W86}
&\int_{-\tau}^{s}\int_{B_{\rho}}(\partial_{t}u_{h}\varphi w_{1}+w_{2}[\mathbf{a}(x,t,u,\nabla u)]_{h}\cdot\nabla\varphi)dxdt\notag\\
&=\int_{-\tau}^{s}\int_{B_{\rho}}[b(x,t,u,\nabla u)]_{h}\varphi w_{2}dxdt.
\end{align}
To begin with, integrating by parts and using Lemma \ref{LEMA90}, we have
\begin{align*}
&\int_{-\tau}^{s}\int_{B_{\rho}}\partial_{t}u_{h}\varphi w_{1}=\int_{B_{\rho}\times(-\tau,s)}\partial_{t}\psi^{2}(u_{h})\zeta^{p}w_{1}\notag\\
&=\int_{B_{\rho}}\psi^{2}(u_{h})(x,s)\zeta^{p}(x)w_{1}dx-\int_{B_{\rho}}\psi^{2}(u_{h})(x,-\tau)\zeta^{p}(x)w_{1}dx\notag\\
&\rightarrow\int_{B_{\rho}}\psi^{2}(u)(x,s)\zeta^{p}(x)w_{1}dx-\int_{B_{\rho}}\psi^{2}(u)(x,-\tau)\zeta^{p}(x)w_{1}dx,\quad\text{as }h\rightarrow0.
\end{align*}
For simplicity, denote
\begin{align*}
Q_{s,k}:=(B_{\rho}\times(-\tau,s))\cap\{(u-k)_{\pm}>0\}.
\end{align*}
As for the remaining two terms in \eqref{W86}, by first letting $h\rightarrow0$, it then follows from Lemma \ref{LEMA90}, $\mathrm{(}\mathbf{H1}\mathrm{)}$--$\mathrm{(}\mathbf{H3}\mathrm{)}$ and Young's inequality that
\begin{align*}
&\int_{-\tau}^{s}\int_{B_{\rho}}[\mathbf{a}(x,t,u,\nabla u)]_{h}\cdot\nabla\varphi w_{2}dxdt\notag\\
&\rightarrow\int_{B_{\rho}\times(-\tau,s)}\mathbf{a}(x,t,u,\nabla u)\cdot\big([\psi^{2}(u)]''\nabla u\zeta^{p}+p[\psi^{2}(u)]'\zeta^{p-1}\nabla\zeta\big) w_{2}dxdt\notag\\
&\geq2\lambda_{1}\int_{Q_{s,k}}(1+\psi)\psi'^{2}|\nabla u|^{p}\zeta^{p}w_{2}dxdt-2\int_{Q_{s,k}}(1+\psi)\psi'^{2}\phi_{1}\zeta^{p}w_{2}dxdt\notag\\
&\quad-2p\lambda_{2}\int_{Q_{s,k}}\psi|\psi'|\zeta^{p-1}|\nabla\zeta||\nabla u|^{p-1}w_{2}dxdt\notag\\
&\quad-2p\int_{Q_{s,k}}\psi|\psi'|\zeta^{p-1}|\nabla\zeta|\phi_{2}w_{2}dxdt\notag\\
&\geq\lambda_{1}\int_{Q_{s,k}}(1+\psi)\psi'^{2}|\nabla u|^{p}\zeta^{p}w_{2}dxdt-2\int_{Q_{s,k}}(1+\psi)\psi'^{2}\phi_{1}\zeta^{p}w_{2}dxdt\notag\\
&\quad-C\int_{Q_{s,k}}\psi|\psi'|^{2-p}|\nabla\zeta|^{p}w_{2}dxdt-C\int_{Q_{s,k}}\psi\psi'^{2}\phi_{2}^{\frac{p}{p-1}}\zeta^{p}w_{2}dxdt,
\end{align*}
and
\begin{align*}
&\int_{-\tau}^{s}\int_{B_{\rho}}[b(x,t,u,\nabla u)]_{h}\varphi w_{2}dxdt\rightarrow\int_{B_{\rho}\times(-\tau,s)}b(x,t,u,\nabla u)[\psi^{2}(u)]'\zeta^{p}w_{2}dxdt\notag\\
&\leq2\lambda_{3}\int_{Q_{s,k}}\psi|\psi'||\nabla u|^{p-1}\zeta^{p}w_{2}dxdt+2\int_{Q_{s,k}}\psi|\psi'|\zeta^{p}\phi_{3}w_{2}dxdt\notag\\
&\leq\frac{2\lambda_{1}}{3}\int_{Q_{s,k}}(1+\psi)\psi'^{2}\zeta^{p}|\nabla u|^{p}w_{2}dxdt+C\int_{Q_{s,k}}\psi|\psi'|^{2-p}|\zeta|^{p}w_{2}dxdt\notag\\
&\quad+\frac{2}{\delta}\ln\left(\frac{H_{k}^{\pm}}{\delta}\right)\int_{Q_{s,k}}\phi_{3}\zeta^{p}w_{2}dxdt,
\end{align*}
where we also used the facts of $\psi\leq\ln\frac{H_{k}^{\pm}}{\delta}$ and $|\psi'|\leq\frac{1}{\delta}$. Combining these above facts, we have from H\"{o}lder's inequality that
\begin{align*}
&\sup\limits_{-\tau<t<0}\int_{B_{\rho}}\Psi^{\pm}_{k,\delta}(u)(x,t)\zeta^{p}(x)w_{1}dx\notag\\
&\leq\int_{B_{\rho}}\Psi^{\pm}_{k,\delta}(u)(x,-\tau)\zeta^{p}(x)w_{1}dx+C\int_{Q(\rho,\tau)}\Psi^{\pm}_{k,\delta}(u)|\partial_{u}\Psi^{\pm}_{k,\delta}(u)|^{2-p}(|\nabla\zeta|^{p}+|\zeta|^{p})w_{2}\notag\\
&\quad+\frac{C}{\delta^{2}}\left(1+\ln\frac{H_{k}^{\pm}}{\delta}\right)\int_{Q(\rho,\tau)\cap\{(u-k)_{\pm}>0\}}\phi w_{2}dxdt\notag\\
&\leq\int_{B_{\rho}}\Psi^{\pm}_{k,\delta}(u)(x,-\tau)\zeta^{p}(x)w_{1}dx+C\int_{Q(\rho,\tau)}\Psi^{\pm}_{k,\delta}(u)|\partial_{u}\Psi^{\pm}_{k,\delta}(u)|^{2-p}(|\nabla\zeta|^{p}+|\zeta|^{p})w_{2}\notag\\
&\quad+\frac{C}{\delta^{2}}\left(1+\ln\frac{H_{k}^{\pm}}{\delta}\right)\|\phi\|_{L^{l_{0}}(B_{1}\times(-1,0),w_{2})}|Q(\rho,\tau)\cap \{(u-k)_{\pm}>0\}|_{\nu_{w_{2}}}^{1-\frac{1}{l_{0}}}.
\end{align*}
The proof is complete.

\end{proof}

We are now ready to use the logarithmic estimates in Lemma \ref{lem8021} to complete the proof of Proposition \ref{LEM30}.
\begin{proof}[Proof of Proposition \ref{LEM30}]
Let $k=\mu^{+}-\frac{\omega}{2}$ and take $\delta=\frac{\omega}{2^{j_{0}+1}}$ in \eqref{P998} for some positive constant $j_{0}>2$ to be determined later. For brevity, write
\begin{align*}
\Psi:=\ln^{+}\frac{H_{k}^{+}}{H_{k}^{+}-(u-(\mu^{+}-\frac{\omega}{2}))_{+}+\frac{\omega}{2^{j_{0}+1}}},
\end{align*}
where
\begin{align*}
H_{k}^{+}=\sup\limits_{Q_{R}^{+}(\omega)}\left(u-\left(\mu^{+}-\frac{\omega}{2}\right)\right)_{+}\leq\frac{\omega}{2},\quad Q_{R}^{+}(\omega):=B_{R}\times[\bar{t},\bar{t}+\omega^{2-p}R^{p+\vartheta}].
\end{align*}
Let $\zeta\in C^{\infty}(B_{R})$ be a smooth cutoff function satisfying that $\zeta=1$ in $B_{(1-\varrho)R}$, $0\leq\zeta\leq1$ and $|\nabla\zeta|\leq\frac{4}{\varrho R}$ in $B_{R}$, where $\varrho\in(0,1)$ is to be chosen later. Observe that
\begin{align*}
\Psi\leq\ln\left(\frac{\frac{\omega}{2}}{\frac{\omega}{2^{j_{0}+1}}}\right)=j_{0}\ln2,
\end{align*}
and
\begin{align*}
|\partial_{u}\Psi|^{2-p}=\Big|H_{k}^{+}-(u-k)_{+}+\frac{\omega}{2^{j_{0}+1}}\Big|^{p-2}\leq\omega^{p-2}.
\end{align*}
Based on these above facts, it follows from \eqref{QE010}, Lemmas \ref{LEM860} and \ref{lem8021} that for any $t\in[\bar{t},\bar{t}+\omega^{2-p}R^{p+\vartheta}]$, if $\omega>A_{1}R^{\varepsilon_{0}}$ with $A_{1}:=4^{\frac{j_{0}l_{0}}{p(l_{0}-1)+2}}$,
\begin{align}\label{E90}
&\int_{B_{(1-\varrho)R}}\Psi^{2}(x,t)w_{1}dx\leq\int_{B_{R}}\Psi^{2}(x,\bar{t})\zeta(x)w_{1}dx+\frac{C}{R^{p}}\int_{Q_{R}^{+}(\omega)}\Psi|\partial_{u}\Psi|^{2-p}w_{2}dxdt\notag\\
&\quad+C\Big(\frac{2^{j_{0}}}{\omega}\Big)^{2}\left(1+\ln\frac{H_{k}^{+}2^{j_{0}}}{\omega}\right)\Big|Q_{R}^{+}(\omega)\cap \Big\{u>\mu^{+}-\frac{\omega}{2}\Big\}\Big|_{\nu_{w_{2}}}^{1-\frac{1}{l_{0}}}\notag\\
&\leq \frac{(j_{0}\ln2)^{2}}{2}|B_{R}|_{\mu_{w_{1}}}+ Cj_{0}\Big(\varrho^{-p}+4^{j_{0}}R^{\frac{p(l_{0}-1)-n-\theta_{1}-\theta_{2}}{l_{0}}}\omega^{-\frac{p(l_{0}-1)+2}{l_{0}}}\Big)|B_{R}|_{\mu_{w_{1}}}\notag\\
&\leq (j_{0}\ln2)^{2}|B_{R}|_{\mu_{w_{1}}}\left(\frac{1}{2}+\frac{C}{j_{0}\varrho^{p}}\right).
\end{align}
Note that for $(x,t)\in\{x\in B_{(1-\varrho)R}:u>\mu^{+}-2^{-(j_{0}+1)}\omega\}\times[\bar{t},\bar{t}+\omega^{2-p}R^{p+\vartheta}]$,
\begin{align*}
\Psi^{2}(x,t)\geq\ln^{2}\frac{H_{k}^{+}}{H_{k}^{+}-\frac{\omega}{2}+\frac{\omega}{2^{j_{0}}}}\geq\ln^{2}\left(\frac{\frac{\omega}{2}}{\frac{\omega}{2^{j_{0}}}}\right)=(j_{0}-1)^{2}(\ln2)^{2},
\end{align*}
where we also utilized the fact that $\ln\frac{H_{k}^{+}}{H_{k}^{+}-\frac{\omega}{2}+\frac{\omega}{2^{j_{0}}}}$ is decreasing in $H_{k}^{+}$. Then we deduce that for any $t\in[\bar{t},\bar{t}+\omega^{2-p}R^{p+\vartheta}]$,
\begin{align*}
\int_{B_{(1-\varrho)R}}\Psi^{2}(x,t)w_{1}dx\geq(j_{0}-1)^{2}(\ln2)^{2}\Big|B_{(1-\varrho)R}\cap\Big\{u(\cdot,t)>\mu^{+}-\frac{\omega}{2^{j_{0}+1}}\Big\}\Big|_{\mu_{w_{1}}}.
\end{align*}
Substituting this into \eqref{E90}, we obtain that for any $t\in[\bar{t},\bar{t}+\omega^{2-p}R^{p+\vartheta}]$,
\begin{align*}
&\Big|B_{(1-\varrho)R}\cap\Big\{u(\cdot,t)>\mu^{+}-\frac{\omega}{2^{j_{0}+1}}\Big\}\Big|_{\mu_{w_{1}}}\leq\left(\frac{j_{0}}{j_{0}-1}\right)^{2}|B_{R}|_{\mu_{w_{1}}}\left(\frac{1}{2}+\frac{C}{j_{0}\varrho^{p}}\right).
\end{align*}
Observe that $|B_{R}\setminus B_{(1-\varrho)R}|_{\mu_{w_{1}}}\leq C\varrho|B_{R}|_{\mu_{w_{1}}}$. Then by choosing $\varrho=j_{0}^{-\frac{1}{p+1}}$ and $j_{0}=(24\overline{C}_{\ast})^{p+1}$, we deduce that for any $t\in[\bar{t},\bar{t}+\omega^{2-p}R^{p+\vartheta}]$,
\begin{align*}
&\frac{\big|B_{R}\cap\big\{u(\cdot,t)>\mu^{+}-\frac{\omega}{2^{j_{0}+1}}\big\}\big|_{\mu_{w_{1}}}}{|B_{R}|_{\mu_{w_{1}}}}\notag\\
&\leq\frac{j_{0}^{2}}{2(j_{0}-1)^{2}}+\frac{\overline{C}_{\ast}}{\varrho^{p}}\left(\varrho^{p+1}+\frac{1}{j_{0}}\right)\leq\frac{j_{0}^{2}}{2(j_{0}-1)^{2}}+\frac{2\overline{C}_{\ast}}{\sqrt[p+1]{j_{0}}}\leq\frac{3}{4}.
\end{align*}

\end{proof}

%\noindent{\bf{\large Data Availability Statement.}} The data used to support the findings of this study are available from the corresponding author upon request.

\noindent{\bf{\large Acknowledgements.}} This work was partially supported by the National Key research and development program of
China (No. 2022YFA1005700 and 2020Y-FA0712903). C. Miao was partially supported by the National Natural Science Foundation of China (No. 12026407 and 12071043). Z. Zhao was partially supported by China Postdoctoral Science Foundation (No. 2021M700358).

%\noindent{\bf{\large Statements.}}

%\noindent{\bf{\large Acknowledgements.}}


\begin{thebibliography}{99}

%\bibitem{A2021} G. Akagi, Rates of convergence to non-degenerate asymptotic profiles for fast diffusion via energy methods. arXiv:2109.03960.

\bibitem{AL1983} H.W. Alt and S. Luckhaus, Quasilinear elliptic-parabolic differential equations. Math. Z. 183 (1983), no. 3, 311-341.

%\bibitem{BCG2006} H. Bahouri, J.-Y. Chemin and I. Gallagher, Refined Hardy inequalities. Ann. Sc. Norm. Super. Pisa Cl. Sci. (5) 5 (2006), no. 3, 375-391.

\bibitem{BH1980} J.B. Berryman and C.J. Holland, Stability of the separable solution for fast diffusion. Arch. Rational Mech. Anal. 74 (1980), 379-388.

%\bibitem{BDMS2018} V. B\"{o}gelein, F. Duzaar, P. Marcellini and C. Scheven, Doubly nonlinear equations of porous medium type. Arch. Ration. Mech. Anal. 229 (2018), no. 2, 503-545.
%
%\bibitem{BDMS201802} V. B\"{o}gelein, F. Duzaar, P. Marcellini and C. Scheven, A variational approach to doubly nonlinear equations. Atti Accad. Naz. Lincei Rend. Lincei Mat. Appl. 29 (2018), no. 4, 739-772.
%
%\bibitem{BDL2021} V. B\"{o}gelein, F. Duzaar and N. Liao, On the H\"{o}lder regularity of signed solutions to a doubly nonlinear equation. J. Funct. Anal. 281 (2021), no. 9, Paper No. 109173, 58 pp.
%
\bibitem{BDGLS2023} V. B\"{o}gelein, F. Duzaar, U. Gianazza, N. Liao and C. Scheven, H\"{o}lder Continuity of the Gradient of Solutions to Doubly Non-Linear Parabolic Equations. arXiv:2305.08539.

%\bibitem{BF2021} M. Bonforte and A. Figalli, Sharp extinction rates for fast diffusion equations on generic bounded domains. Comm. Pure Appl. Math. 74 (2021), no. 4, 744-789.
%
%\bibitem{BGV2012} M. Bonforte, G. Grillo and J.L. V$\acute{\mathrm{a}}$zquez, Behaviour near extinction for the Fast Diffusion Equation on bounded domains. J. Math. Pures Appl. (9) 97 (2012), no. 1, 1-38.
%
%\bibitem{BV2010} M. Bonforte, J.L. V$\acute{\mathrm{a}}$zquez, Positivity, local smoothing, and Harnack inequalities for very fast diffusion equations. Adv. Math. 223 (2010), no. 2, 529-578.

\bibitem{BSS2022} M. Bonforte, N. Simonov and D. Stan, The Cauchy problem for the fast $p$-Laplacian evolution equation. Characterization of the global Harnack principle and fine asymptotic behaviour. J. Math. Pures Appl. (9) 163 (2022), 83-131.

\bibitem{BS2023} M. Bonforte and N. Simonov, Fine properties of solutions to the Cauchy problem for a fast diffusion equation with Caffarelli-Kohn-Nirenberg weights. Ann. Inst. H. Poincar\'{e} C Anal. Non Lin\'{e}aire 40 (2023), no. 1, 1-59.


%\bibitem{CKN1984} L. Caffarelli, R. Kohn and L. Nirenberg, First order interpolation inequalities with weights, Composition Math. 53 (1984), 259-275.
%
%\bibitem{CR2013} X. Cabr\'{e} and X. Ros-Oton, Sobolev and isoperimetric inequalities with monomial weights. J. Differential Equations 255 (2013), no. 11, 4312-4336.
%
%\bibitem{CRS2016} X. Cabr\'{e}, X. Ros-Oton and J. Serra, Sharp isoperimetric inequalities via the ABP method. J. Eur. Math. Soc. (JEMS) 18 (2016), no. 12, 2971-2998.

\bibitem{CD1988} Y.Z. Chen and E. DiBenedetto, On the local behavior of solutions of singular parabolic equations. Arch. Rational Mech. Anal. 103 (1988), no. 4, 319-345.

\bibitem{C1991} H.J. Choe, H\"{o}lder regularity for the gradient of solutions of certain singular parabolic systems. Comm. Partial Differential Equations 16 (1991), no. 11, 1709-1732.

%\bibitem{CS1984} F. Chiarenza and R.P. Serapioni, A Harnack inequality for degenerate parabolic equations. Comm. Partial Differential Equations 9 (1984), no. 8, 719-749.
%
%\bibitem{CW1985} S. Chanillo and R.L. Wheeden, Weighted Poincar\'{e} and Sobolev inequalities and estimates for weighted Peano maximal functions. Amer. J. Math. 107 (1985), no. 5, 1191-1226.

%\bibitem{CW2006} S.-K. Chua and R.L. Wheeden, Estimates of best constants for weighted Poincar\'{e} inequalities on convex domains. Proc. London Math. Soc. (3) 93 (2006), no. 1, 197-226.

\bibitem{DK2007} P. Daskalopoulos and C. Kenig, Degenerate diffusions. Initial value problems and local regularity theory. EMS Tracts in Mathematics, 1. European Mathematical Society (EMS), Z\"{u}rich, 2007.


%\bibitem{D1957} E. De Giorgi, Sulla differenziabilit$\grave{\mathrm{a}}$ e l'analiticit$\grave{\mathrm{a}}$ delle estremali degli integrali multipli regolari. Mem. Accad. Sci. Torino. Cl. Sci. Fis. Mat. Nat. (3) 3 (1957), 25-43.


\bibitem{DLY2021} H.J. Dong, Y.Y. Li and Z.L. Yang, Optimal gradient estimates of solutions to the insulated conductivity problem in dimension greater than two. arXiv:2110.11313, J. Eur. Math. Soc., to appear.

\bibitem{DLY2022} H.J. Dong, Y.Y. Li and Z.L. Yang, Gradient estimates for the insulated conductivity problem: the non-umbilical case. arXiv:2203.10081.

\bibitem{DP2023} H.J. Dong and T. Phan, Weighted mixed-norm $L_{p}$ estimates for equations in non-divergence form with singular coefficients: the Dirichlet problem. J. Funct. Anal. 285 (2023), no. 2, Paper No. 109964, 43 pp.

\bibitem{DPT2023} H.J. Dong, T. Phan and H.V. Tran, Degenerate linear parabolic equations in divergence form on the upper half space. Trans. Amer. Math. Soc. 376 (2023), no. 6, 4421-4451.

\bibitem{DPT202302} H.J. Dong, T. Phan and H.V. Tran, Nondivergence form degenerate linear parabolic equations on the upper half space. arXiv:2306.11567.

\bibitem{D1983} E. DiBenedetto, $C^{1+\alpha}$ local regularity of weak solutions of degenerate elliptic equations. Nonlinear Anal. 7 (1983), no. 8, 827-850.

\bibitem{DF1984} E. DiBenedetto and A. Friedman, Regularity of solutions of nonlinear degenerate parabolic systems. J. Reine Angew. Math. 349 (1984), 83-128.

\bibitem{DF198501} E. DiBenedetto and A. Friedman, H\"{o}lder estimates for nonlinear degenerate parabolic systems. J. Reine Angew. Math. 357 (1985), 1-22.

\bibitem{DF198502} E. DiBenedetto and A. Friedman, Addendum to: ``H\"{o}lder estimates for nonlinear degenerate parabolic systems''. J. Reine Angew. Math. 363 (1985), 217-220.

\bibitem{D1986} E. DiBenedetto, On the local behaviour of solutions of degenerate parabolic equations with measurable coefficients. Ann. Scuola Norm. Sup. Pisa Cl. Sci. (4) 13 (1986), no. 3, 487-535.

\bibitem{DK1992} E. DiBenedetto and Y.C. Kwong, Harnack estimates and extinction profile for weak solution of certain singular parabolic equations. Trans. Amer. Math. Soc. 330 (2) (1992), 783-811.

\bibitem{D1993} E. DiBenedetto, Degenerate parabolic equations. Universitext. Springer-Verlag, New York, 1993.

\bibitem{DGV2008} E. DiBenedetto, U. Gianazza and V. Vespri, Harnack estimates for quasi-linear degenerate parabolic differential equations. Acta Math. 200 (2008), no. 2, 181-209.

\bibitem{DGV200802} E. DiBenedetto, U. Gianazza and V. Vespri, Subpotential lower bounds for nonnegative solutions to certain quasi-linear degenerate parabolic equations. Duke Math. J. 143 (2008), no. 1, 1-15.

\bibitem{DGV2012} E. DiBenedetto, U. Gianazza and V. Vespri, Harnack's inequality for degenerate and singular parabolic equations. Springer Monographs in Mathematics. Springer, New York, 2012.


%\bibitem{DKV1991} E. DiBenedetto, Y.C. Kwong and V. Vespri, Local space-analyticity of solutions of certain singular parabolic equations. Indiana Univ. Math. J. 40 (2) (1991), 741-765.

\bibitem{E1982} L.C. Evans, A new proof of local $C^{1,\alpha}$ regularity for solutions of certain degenerate elliptic P.D.E.
J. Differential Equations 45 (1982), no. 3, 356-373.

\bibitem{FKS1982} E.B. Fabes, C.E. Kenig and R.P. Serapioni, The local regularity of solutions of degenerate elliptic equations. Comm. Partial Differential Equations 7 (1982), no. 1, 77-116.

\bibitem{FP2013} P.M.N. Feehan and C.A. Pop, A Schauder approach to degenerate-parabolic partial differential equations with unbounded coefficients. J. Differential Equations 254 (2013), no. 12, 4401-4445.

%\bibitem{FS2000} E. Feireisl and F. Simondon, Convergence for semilinear degenerate parabolic equations in several space dimension. J. Dynam. Differential Equations 12 (2000), 647-673.
%
%\bibitem{GG1982} M. Giaquinta and E. Giusti, On the regularity of the minima of variational integrals. Acta Math. 148 (1982), 31-46.

\bibitem{GW1990} C.E. Guti\'{e}rrez and R.L. Wheeden, Mean value and Harnack inequalities for degenerate parabolic equations. Colloq. Math. 60/61 (1990), no. 1, 157-194.

\bibitem{GW1991} C.E. Guti\'{e}rrez and R.L. Wheeden, Harnack's inequality for degenerate parabolic equations. Comm. Partial Differential Equations 16 (1991), no. 4-5, 745-770.

%\bibitem{G2014} L. Grafakos, Classical Fourier analysis. Third edition. Graduate Texts in Mathematics, 249. Springer, New York, 2014.

\bibitem{HKM2006} J. Heinonen, T. Kilpel\"{a}inen and O. Martio, Nonlinear potential theory of degenerate elliptic equations. Unabridged republication of the 1993 original. Dover Publications, Inc., Mineola, NY, 2006.

\bibitem{JX2019} T.L. Jin and J.G. Xiong, Optimal boundary regularity for fast diffusion equations in bounded domains. Amer. J. Math. 145 (2023), no. 1, 151-219.

\bibitem{JX2022} T.L. Jin and J.G. Xiong, Regularity of solutions to the Dirichlet problem for fast diffusion equations. arXiv:2201.10091.

\bibitem{JRX2023} T.L. Jin, X. Ros-Oton and J.G. Xiong, Optimal regularity and fine asymptotics for the porous medium equation in bounded domains. arXiv:2211.06124.

\bibitem{JX2023} T.L. Jin and J.G. Xiong, H\"{o}lder regularity for the linearized porous medium equation in bounded domains. arXiv:2303.00321.

%\bibitem{K1988} Y.C. Kwong, Interior and boundary regularity of solutions to a plasma type equation. Proc. Amer. Math. Soc. 104 (1988), no. 2, 472-478.
%
%\bibitem{L1986} C.S. Lin, Interpolation inequalities with weights. Comm. Partial Differential Equations 11 (1986), no. 14, 1515-1538.

\bibitem{L1994} G.M. Lieberman, Gradient estimates for a new class of degenerate elliptic and parabolic equations. Ann. Scuola Norm. Sup. Pisa Cl. Sci. (4) 21 (1994), no. 4, 497-522.

\bibitem{L2019} P. Lindqvist, Notes on the stationary $p$-Laplace equation. SpringerBriefs in Mathematics. Springer, Cham, 2019.

%\bibitem{LL2022} E. Lindgren and P. Lindqvist, On a comparison principle for Trudinger's equation. Adv. Calc. Var. 15 (2022), no. 3, 401-415.


\bibitem{LLY201801} L. Li, Y.Y. Li, and X.K. Yan, Homogeneous solutions of stationary Navier-Stokes equations with isolated singularities on the unit sphere. I. One singularity. Arch. Ration. Mech. Anal. 227 (2018), no. 3, 1091-1163.

\bibitem{LLY201802} L. Li, Y.Y. Li, and X.K. Yan, Homogeneous solutions of stationary Navier-Stokes equations with isolated singularities on the unit sphere. II. Classification of axisymmetric no-swirl solutions. J. Differential Equations 264 (2018), no. 10, 6082-6108.

\bibitem{LY202100} Y.Y. Li and X.K. Yan, Asymptotic stability of homogeneous solutions of incompressible stationary Navier-Stokes equations.
J. Differential Equations 297 (2021), 226-245.

\bibitem{LY2023} Y.Y. Li and X.K. Yan, Anisotropic Caffarelli-Kohn-Nirenberg type inequalities. Adv. Math. 419 (2023), Paper No. 108958, 44 pp.

%
%\bibitem{LSU1968} O.A. Lady\u{z}enskaja, V.A. Solonnikov and N.N. Ural'ceva, Linear and quasilinear equations of parabolic type. Translated from the Russian by S. Smith. Translations of Mathematical Monographs, Vol. 23. American Mathematical Society, Providence, RI, 1968.
%
%\bibitem{M1961} J. Moser, On Harnack's theorem for elliptic differential equations. Comm. Pure Appl. Math. 14 (1961), 577-591.
%
%\bibitem{M1964} J. Moser, A Harnack inequality for parabolic differential equations. Comm. Pure Appl. Math. 17 (1964), 101-134.

\bibitem{M1972} B. Muckenhoupt, Benjamin, Weighted norm inequalities for the Hardy maximal function. Trans. Amer. Math. Soc. 165 (1972), 207-226.

\bibitem{MZ2023} C.X. Miao and Z.W. Zhao, Local regularity for nonlinear elliptic and parabolic equations with anisotropic weights. Proc. Edinb. Math. Soc. (2) 66 (2023), no. 2, 391-436.

\bibitem{MZ202302} C.X. Miao and Z.W. Zhao, On a class of anisotropic Muckenhoupt weights and their applications to $p$-Laplace equations. arXiv:2310.01359.

%\bibitem{NS2018} H. Nguyen and M. Squassina, Fractional Caffarelli-Kohn-Nirenberg inequalities. J. Funct. Anal. 274 (2018), no. 9, 2661-2672.
%
%\bibitem{NS2019} H. Nguyen and M. Squassina, On Hardy and Caffarelli-Kohn-Nirenberg inequalities. J. Anal. Math. 139 (2019), no. 2, 773-797.

\bibitem{STV2021} Y. Sire, S. Terracini and S. Vita, Liouville type theorems and regularity of solutions to degenerate or singular problems part I: even solutions. Comm. Partial Differential Equations 46 (2021), no. 2, 310-361.

\bibitem{STV202102} Y. Sire, S. Terracini and S. Vita, Liouville type theorems and regularity of solutions to degenerate or singular problems part II: odd solutions. Math. Eng. 3 (2021), no. 1, Paper No. 5, 50 pp.

%\bibitem{S1983} P.E. Sacks, Continuity of solutions of a singular parabolic equation. Nonlinear Anal. 7 (1983), no. 4, 387-409.

%\bibitem{S1964} J. Serrin, Local behavior of solutions of quasi-linear equations. Acta Math. 111 (1964), 247-302.

\bibitem{S2010} M. Surnachev, A Harnack inequality for weighted degenerate parabolic equations. J. Differential Equations 248 (2010), no. 8, 2092-2129.

%\bibitem{T1967} N.S. Trudinger, On Harnack type inequalities and their application to quasilinear elliptic equations. Comm. Pure Appl. Math. 20 (1967), 721-747.
%
%\bibitem{T1968} N.S. Trudinger, Pointwise estimates and quasilinear parabolic equations. Comm. Pure Appl. Math. 21 (1968), 205-226.

\bibitem{U2008} J.M. Urbano, The method of intrinsic scaling. A systematic approach to regularity for degenerate and singular PDEs. Lecture Notes in Mathematics, 1930. Springer-Verlag, Berlin, 2008.

\bibitem{V2007} J.L. V\'{a}zquez, The porous medium equation. Mathematical theory. Oxford Mathematical Monographs. The Clarendon Press, Oxford University Press, Oxford, 2007.



\end{thebibliography}
\end{document}